\documentclass[a4paper]{article}

\DeclareFontShape{OT1}{cmr}{m}{scit}{<->ssub*cmr/m/sc}{}
\usepackage{slantsc}

\usepackage[margin=1in]{geometry} 

\usepackage{amsmath}
\usepackage{amsthm}
\usepackage{aliascnt}
\usepackage{amssymb}
\usepackage{graphicx}
\usepackage{xcolor}
\usepackage{mathrsfs} 
\usepackage[utf8]{inputenc}
\usepackage{adjustbox}
\usepackage{placeins}
\usepackage{enumerate}
\usepackage{verbatim}

\usepackage{tikz}
\usetikzlibrary{shapes,calc,math,backgrounds,matrix}

\usepackage{hyperref}
\hypersetup{
	colorlinks,
	allcolors=blue,
	unicode,
	bookmarksopen,
	bookmarksdepth=2,
	pdftitle={On Realizing Reconfiguration Graphs of Cliques},
	pdfauthor={Duc A. Hoang},
	pdfkeywords={Reconfiguration graph, Graph realization, Clique, Token sliding, Token jumping, Johnson graph}
}

\usepackage[capitalise,noabbrev]{cleveref} 
\crefname{cond}{Condition}{Conditions}
\creflabelformat{cond}{(#2C#1#3)}
\crefname{eitem}{}{}
\creflabelformat{eitem}{(#2#1#3)}
\crefname{step}{\textbf{Step}}{\textbf{Steps}}
\creflabelformat{step}{{\bf #2#1#3}}

\usepackage[%
backend=biber, 
bibstyle=numeric, 
citestyle=numeric-comp,
maxcitenames=3,
maxbibnames=20,
sorting=ydnt
]{biblatex}
\theoremstyle{plain}
\newtheorem{theorem}{Theorem}[section]
\newaliascnt{corollary}{theorem}
\newtheorem{corollary}[corollary]{Corollary}
\aliascntresetthe{corollary}
\newaliascnt{lemma}{theorem}
\newtheorem{lemma}[lemma]{Lemma}
\aliascntresetthe{lemma}
\newtheorem{claim}{Claim}[theorem]
\newaliascnt{axiom}{theorem}

\aliascntresetthe{axiom}
\newaliascnt{conjecture}{theorem}

\aliascntresetthe{conjecture}
\newaliascnt{fact}{theorem}

\aliascntresetthe{fact}
\newaliascnt{hypothesis}{theorem}

\aliascntresetthe{hypothesis}
\newaliascnt{assumption}{theorem}

\aliascntresetthe{assumption}
\newaliascnt{proposition}{theorem}
\newtheorem{proposition}[proposition]{Proposition}
\aliascntresetthe{proposition}
\newaliascnt{criterion}{theorem}

\aliascntresetthe{criterion}
\theoremstyle{definition}
\newaliascnt{definition}{theorem}

\aliascntresetthe{definition}
\newaliascnt{example}{theorem}

\aliascntresetthe{example}
\newaliascnt{remark}{theorem}

\aliascntresetthe{remark}
\newaliascnt{problem}{theorem}

\aliascntresetthe{problem}
\newaliascnt{principle}{theorem}

\aliascntresetthe{principle}

\crefname{theorem}{Theorem}{Theorems}
\Crefname{theorem}{Theorem}{Theorems}
\crefname{corollary}{Corollary}{Corollaries}
\Crefname{corollary}{Corollary}{Corollaries}
\crefname{lemma}{Lemma}{Lemmas}
\Crefname{lemma}{Lemma}{Lemmas}
\crefname{claim}{Claim}{Claims}
\Crefname{claim}{Claim}{Claims}
\crefname{axiom}{Axiom}{Axioms}
\Crefname{axiom}{Axiom}{Axioms}
\crefname{conjecture}{Conjecture}{Conjectures}
\Crefname{conjecture}{Conjecture}{Conjectures}
\crefname{fact}{Fact}{Facts}
\Crefname{fact}{Fact}{Facts}
\crefname{hypothesis}{Hypothesis}{Hypotheses}
\Crefname{hypothesis}{Hypothesis}{Hypotheses}
\crefname{assumption}{Assumption}{Assumptions}
\Crefname{assumption}{Assumption}{Assumptions}
\crefname{proposition}{Proposition}{Propositions}
\Crefname{proposition}{Proposition}{Propositions}
\crefname{criterion}{Criterion}{Criteria}
\Crefname{criterion}{Criterion}{Criteria}
\crefname{definition}{Definition}{Definitions}
\Crefname{definition}{Definition}{Definitions}
\crefname{example}{Example}{Examples}
\Crefname{example}{Example}{Examples}
\crefname{remark}{Remark}{Remarks}
\Crefname{remark}{Remark}{Remarks}
\crefname{problem}{Problem}{Problems}
\Crefname{problem}{Problem}{Problems}
\crefname{principle}{Principle}{Principles}
\Crefname{principle}{Principle}{Principles}

\usepackage[linesnumbered,ruled,vlined,commentsnumbered]{algorithm2e} 

\usepackage[pagewise]{lineno}

\usepackage[colorinlistoftodos]{todonotes}
\definecolor{done}{rgb}{0.55, 0.71, 0.0}
\definecolor{highpriority}{rgb}{0.82, 0.1, 0.26}
\definecolor{lowpriority}{rgb}{1.0, 0.75, 0.0}
\definecolor{mediumpriority}{rgb}{1.0, 0.49, 0.0}
\newcommand{\Mcomment}[3]{%
\ifsubmission%
	\else%
	{\color{#1}\bfseries\sffamily(#3)%
	}%
	\marginpar{\textcolor{#1}{\bfseries\sffamily #2}}%
	\fi%
}
\newcommand{\TaskNewPerson}[3]{%
	\expandafter\newcommand\csname #2\endcsname[1]{\Mcomment{#3}{#2}{##1}}%
	\expandafter\newcommand\csname Task#1\endcsname[2][mediumpriority]{\expandafter\TaskPerson[backgroundcolor=##1]{#1}{##2}}
	}

\newcommand*{\TaskPerson}[3][]{
\ifsubmission%
	\else%
	\todo[inline,#1]{\textbf{(#2)} #3}
	\fi%
}

\newcommand{\email}[1]{\href{mailto:#1}{\texttt{#1}}} 
\newcommand{\sfTS}{\mathsf{TS}} 
\newcommand{\sfTJ}{\mathsf{TJ}} 
\newcommand{\TS}[2]{\sfTS_{#1}(#2)}
\newcommand{\TJ}[2]{\sfTJ_{#1}(#2)}
\newcommand{\KTS}{\mathcal{K}^{\sfTS}}
\newcommand{\KTJ}{\mathcal{K}^{\sfTJ}}
\newcommand{\Int}{\mathsf{Int}}
\newcommand{\CP}[1]{\text{CP}_{#1}} 

\title{\textbf{On Realizing Reconfiguration Graphs of Cliques}}
\author{Duc~A.~Hoang$^1$}

\date{
	$^1$VNU University of Science,
	Vietnam National University,
	Hanoi, 
	Vietnam \\ \email{hoanganhduc@hus.edu.vn}
	\\[2ex]%
}

\newif\ifsubmission
\submissiontrue

\begin{document}
\maketitle

\tableofcontents

\begin{abstract}
	For a graph $H$ and an integer $k\ge 1$, the \emph{Token Sliding reconfiguration graph} $\mathsf{TS}_k(H)$ and the \emph{Token Jumping reconfiguration graph} $\mathsf{TJ}_k(H)$ have as vertices the $k$-cliques of $H$, with two vertices adjacent when one clique is obtained from the other by replacing one vertex with an adjacent non-member, and respectively by an arbitrary non-member while keeping a clique. For a target graph $G$, we study the feasibility sets $\mathcal{K}^{\mathsf{TS}}(G)$ and $\mathcal{K}^{\mathsf{TJ}}(G)$, consisting of all integers $k$ for which $G$ is isomorphic to $\mathsf{TS}_k(H)$ and $\mathsf{TJ}_k(H)$, respectively, for some graph $H$. We determine the exact feasibility sets for complete graphs, paths, cycles, complete bipartite graphs, book graphs, friendship graphs, and their complements, and give complete classifications for all nontrivial Johnson graphs.

	\noindent\textbf{Keywords:} Reconfiguration graph, Graph realization, Clique, Token sliding, Token jumping, Johnson graph
\end{abstract}

\section{Introduction}
\label{sec:intro}

Recently, \emph{combinatorial reconfiguration} has attracted a lot of attention in the theoretical computer science communities; see, for example, \cite{Heuvel13,Nishimura18,MynhardtN19,BousquetMNS24} for the well-known surveys in this area. In a reconfiguration setting of a \emph{source} problem $\Pi$ (e.g., \textsc{Independent Set}, \textsc{Clique}, \textsc{Dominating Set}, etc.), a description of a solution to $\Pi$ is called a \emph{configuration}, and two configurations are \emph{adjacent} if one can be obtained from the other by applying a specified \emph{reconfiguration rule} (e.g., token sliding, token jumping, vertex recoloring, etc.). The \emph{reconfiguration graph} of $\Pi$ on an instance $I$ is the graph whose vertices are the configurations of $I$, with two vertices adjacent when the corresponding configurations are adjacent. 

A \emph{clique} of a graph $G$ is a set of pairwise adjacent vertices of $G$.
In this paper, we focus on the source problem \textsc{Clique} and study the reconfiguration graphs of cliques under the \emph{token sliding} ($\sfTS$) and \emph{token jumping} ($\sfTJ$) rules. 
More precisely, for a given target graph $G$, we study the problem of determining for which integers $k$ there exists a graph $H$ whose corresponding Token Sliding/Token Jumping reconfiguration graph is isomorphic to $G$, and respectively denote by $\mathcal{K}^{\mathsf{TS}}(G)$ and $\mathcal{K}^{\mathsf{TJ}}(G)$ the sets of all such integers.
For related background on reconfiguration of colourings and dominating sets, see \cite{MynhardtN19}; for realization problems for reconfiguration graphs of independent sets, see \cite{AvisH23,AvisH24}.

Relevant references on clique reconfiguration include the work of Ito, Ono, and Otachi~\cite{ItoOO23} and Lam, Phan, and Hoang~\cite{LamPH26}. In particular, we use from \cite[Lem.~1 and Thm.~1]{LamPH26} a clique-number formula and a clique-structure lemma for $\TS{k}{H}$; see \cref{sec:TS-structural}. Here we take a realization viewpoint: for a prescribed target graph $G$, determine the integers $k$ for which $G\cong \TS{k}{H}$ or $G\cong \TJ{k}{H}$ for some graph $H$. We determine these feasibility sets for several natural target families. Our main results are as follows.
\begin{itemize}
\item We determine the exact TS feasibility sets for complete graphs, paths, cycles, complete bipartite graphs, book graphs, friendship graphs, and their complements (\cref{thm:warmup-ts,thm:camera-ts-basic-complements}).
\item We determine the exact TJ feasibility sets for the same basic families and their complements (\cref{thm:warmup-tj,thm:camera-tj-basic-complements}).
\item For nontrivial Johnson graphs, we determine both feasibility sets $\KTS(J(n,r))$ and $\KTJ(J(n,r))$ (\cref{thm:camera-ts-johnson,thm:camera-tj-johnson}).
\end{itemize}

The rest of this paper is organized as follows. In \cref{sec:prelim}, we introduce the notation used throughout the paper. \cref{sec:TS} is devoted to the Token Sliding side, and \cref{sec:TJ} is devoted to the Token Jumping side.

\section{Preliminaries}
\label{sec:prelim}

In this section, we introduce some notations and definitions that will be used throughout the paper. 
For graph notations and terminologies not defined here, we refer the reader to \cite{Diestel2017}. 
All graphs in this paper are finite, simple, and undirected.
We use $V(G)$ and $E(G)$ to denote the sets of vertices and edges of a graph $G$, respectively.
The neighborhood of a vertex $v$ in $G$, denoted by $N_G(v)$, is the set $\{w \in V(G) : vw \in E(G)\}$.
The closed neighborhood of $v$ in $G$, denoted by $N_G[v]$, is simply the set $N_G(v) \cup \{v\}$.
By abuse of notation, we sometimes also write $N_G(v)$ for the subgraph induced by the neighborhood of $v$ when no confusion can arise.
We also denote by $\deg_G(v)$ the \emph{degree} of $v$ in $G$, which is the size of $N_G(v)$.
The subscripts are often omitted when the graph is clear from the context.
For a set $X$ and an element $x$, we write $X+x$ for $X\cup\{x\}$ when $x\notin X$, and $X-x$ for $X\setminus\{x\}$ when $x\in X$.
More generally, if $I$ is a finite index set and the elements $x_i$ with $i\in I$ are pairwise distinct and outside $X$, then
\[
X+\sum_{i\in I}x_i:=X\cup\{x_i:i\in I\}.
\]

For a graph $G$, its \emph{line graph} is denoted by $L(G)$ and its \emph{complement} by $\overline{G}$.
If $G$ and $H$ are disjoint graphs, then $G\vee H$ denotes their \emph{join} and $G\sqcup H$ denotes their \emph{disjoint union}.
For a positive integer $a$, the notation $aG$ denotes the disjoint union of $a$ copies of $G$.
We write $G\cong H$ to indicate that $G$ and $H$ are \emph{isomorphic}, that is, there exists a bijection $f:V(G)\to V(H)$ such that $uv\in E(G)$ if and only if $f(u)f(v)\in E(H)$ for every $u,v\in V(G)$.
We write $\omega(G)$ for the clique number of $G$ and $\Delta(G)$ for the maximum degree of $G$.
We write $K_n$ for the \emph{complete graph} on $n$ vertices, $P_n$ for the \emph{path} on $n$ vertices, $C_n$ for the \emph{cycle} on $n$ vertices, and $K_{m,n}$ for the \emph{complete bipartite graph} with bipartition classes of sizes $m$ and $n$.
When these families overlap, we use the conventions $P_1=K_1$, $C_3=K_3$, $K_{1,1}=P_2$, $K_{1,2}=P_3$, and $K_{2,2}=C_4$.
When the vertices of a path or cycle are named by numbers, we use the natural path order and the natural cyclic order, respectively.
In particular, $K_{1,n}$ is the \emph{star} on $n+1$ vertices.
A \emph{book graph} $B_p$ ($p\ge 1$) consists of $p$ triangles $K_3$ sharing a common edge.
A \emph{friendship graph} $F_p$ ($p\ge 1$) consists of $p$ triangles sharing a common vertex.
The \emph{cocktail-party graph} $\CP{p}$ ($p\ge 1$) is the complete $p$-partite graph with all partite sets of size $2$, that is, $\CP{p}=K_{2,\dots,2}$.
The \emph{Johnson graph} $J(n,k)$ has as vertices the $k$-subsets of $[n]=\{1,\dots,n\}$, with two vertices adjacent when their intersection has size $k-1$. 
For every $(k-1)$-subset $S$ of $[n]$, the set of all $k$-subsets of $[n]$ that contain $S$ induces a clique in $J(n,k)$; we call this clique the \emph{Johnson star} with core $S$. 
For every $(k+1)$-subset $U$ of $[n]$, the set of all $k$-subsets of $U$ induces a clique in $J(n,k)$; we call this clique the \emph{Johnson top clique} with top $U$.
For a finite set $X$ and an integer $s\ge 0$, we write $\displaystyle \binom{X}{s}$ for the family of all $s$-element subsets of $X$. We call $J(n,r)$ \emph{nontrivial} when $n\ge 2$ and $1\le r\le n-1$. For every $n$ and $1\le r\le n-1$, the complement map $A\mapsto [n]\setminus A$ induces an isomorphism $J(n,r)\cong J(n,n-r)$; we call this isomorphism the \emph{Johnson duality}.
We use the convention that the empty set is the unique $0$-clique of every graph.

Let $H$ be a graph and let $k\ge 1$.
The \emph{Token Sliding (reconfiguration) graph} $\TS{k}{H}$ is the graph whose vertices are the $k$-cliques of $H$; two vertices $A$ and $B$ are adjacent when there exist vertices $u\in A\setminus B$ and $v\in B\setminus A$ such that $A\setminus B=\{u\}$, $B\setminus A=\{v\}$, and $uv\in E(H)$.
The \emph{Token Jumping (reconfiguration) graph} $\TJ{k}{H}$ is defined on the same vertex set, but $A$ and $B$ are adjacent whenever $|A\cap B|=k-1$.
Thus, $\TS{k}{H}$ is indeed a spanning subgraph of $\TJ{k}{H}$.
Since the $k$-cliques of $K_n$ are exactly the $k$-subsets of $V(K_n)$, one can verify that $\TS{k}{K_n}\cong J(n,k)$ and $\TJ{k}{K_n}\cong J(n,k)$ for $n\ge k\ge 1$.
For a target graph $G$, we define $\KTS(G)=\{k\ge 1: \exists H,\ \TS{k}{H}\cong G\}$ and $\KTJ(G)=\{k\ge 1: \exists H,\ \TJ{k}{H}\cong G\}$.
The difference between token sliding and token jumping is illustrated in \cref{fig:ts-tj-complete-noncomplete}, and the two Johnson clique families are shown in \cref{fig:johnson-star-top}.

\begin{figure}[ht]
\centering
\begin{adjustbox}{max width=\textwidth}
\begin{tikzpicture}[x=1cm,y=1cm]
\tikzset{
  grid/.style={draw=black!50, line width=.45pt},
  head/.style={font=\bfseries\small, align=center},
  celltext/.style={font=\small, align=center},
  v/.style={circle, draw=black!70, fill=white, minimum size=18pt, inner sep=1pt, font=\scriptsize},
  gv/.style={circle, draw=black!70, fill=black!4, minimum size=17pt, inner sep=1pt, font=\scriptsize},
  e/.style={draw=black!70, line width=.55pt}
}

\foreach \x in {0,2.0,4.9,8.3,11.7} {
  \draw[grid] (\x,0) -- (\x,5.3);
}
\foreach \y in {0,2.15,4.3,5.3} {
  \draw[grid] (0,\y) -- (11.7,\y);
}

\node[head] at (1.0,4.8) {case};
\node[head] at (3.45,4.8) {source $H$};
\node[head] at (6.6,4.8) {$\mathsf{TS}_{2}(H)$};
\node[head] at (10.0,4.8) {$\mathsf{TJ}_{2}(H)$};

\node[celltext] at (1.0,3.23) {complete};
\node[celltext] at (1.0,1.08) {non-complete};

\node[gv] (k1) at (2.8,3.0) {$1$};
\node[gv] (k2) at (3.45,3.65) {$2$};
\node[gv] (k3) at (4.1,3.0) {$3$};
\draw[e] (k1) -- (k2) -- (k3) -- (k1);
\node[celltext] at (3.45,2.45) {$K_3$};

\foreach \prefix/\xcenter in {cts/6.6, ctj/10.0} {
  \node[v] (\prefix 12) at (\xcenter-0.65,3.0) {$12$};
  \node[v] (\prefix 13) at (\xcenter,3.65) {$13$};
  \node[v] (\prefix 23) at (\xcenter+0.65,3.0) {$23$};
  \draw[e] (\prefix 12) -- (\prefix 13) -- (\prefix 23) -- (\prefix 12);
}

\node[gv] (p1) at (2.65,1.08) {$1$};
\node[gv] (p2) at (3.45,1.08) {$2$};
\node[gv] (p3) at (4.25,1.08) {$3$};
\draw[e] (p1) -- (p2) -- (p3);
\node[celltext] at (3.45,0.50) {$P_3$};

\node[v] (pts12) at (6.15,1.08) {$12$};
\node[v] (pts23) at (7.05,1.08) {$23$};

\node[v] (ptj12) at (9.55,1.08) {$12$};
\node[v] (ptj23) at (10.45,1.08) {$23$};
\draw[e] (ptj12) -- (ptj23);

\end{tikzpicture}
\end{adjustbox}
\caption{Token sliding and token jumping agree on the complete example $K_3$, but differ on the non-complete example $P_3$. The vertices inside the reconfiguration graphs are the $2$-cliques of the source graph.}
\label{fig:ts-tj-complete-noncomplete}
\end{figure}
\begin{figure}[ht]
\centering
\begin{adjustbox}{max width=\textwidth}
\begin{tikzpicture}[x=1cm,y=1cm]
\tikzset{
  title/.style={font=\bfseries\small, align=center},
  note/.style={font=\small, align=center},
  elem/.style={circle, draw=black!65, fill=white, minimum size=20pt, inner sep=1pt, font=\scriptsize},
  markedElem/.style={elem, draw=black, fill=black!12, line width=.7pt},
  jv/.style={circle, draw=black!70, fill=white, minimum size=21pt, inner sep=1pt, font=\scriptsize},
  markedv/.style={jv, draw=black, fill=black!6, line width=.6pt},
  e/.style={draw=black!55, line width=.5pt}
}

\node[title] at (2.9,3.85) {Johnson star in $J(5,2)$};
\node[title] at (9.9,3.85) {Johnson top in $J(5,2)$};

\node[markedElem] at (1.3,3.15) {$1$};
\node[elem] at (2.1,3.15) {$2$};
\node[elem] at (2.9,3.15) {$3$};
\node[elem] at (3.7,3.15) {$4$};
\node[elem] at (4.5,3.15) {$5$};
\node[note] at (2.9,2.35) {$[5]=\{1,2,3,4,5\}$,\quad core $\{1\}$};

\node[markedv] (s12) at (2.25,1.45) {$12$};
\node[markedv] (s13) at (3.55,1.45) {$13$};
\node[markedv] (s14) at (2.25,0.35) {$14$};
\node[markedv] (s15) at (3.55,0.35) {$15$};
\foreach \a/\b in {s12/s13,s12/s14,s12/s15,s13/s14,s13/s15,s14/s15} {
  \draw[e] (\a) -- (\b);
}

\node[markedElem] at (8.3,3.15) {$1$};
\node[markedElem] at (9.1,3.15) {$2$};
\node[markedElem] at (9.9,3.15) {$3$};
\node[elem] at (10.7,3.15) {$4$};
\node[elem] at (11.5,3.15) {$5$};
\node[note] at (9.9,2.35) {$[5]=\{1,2,3,4,5\}$,\quad top $\{1,2,3\}$};

\node[markedv] (t12) at (9.9,1.45) {$12$};
\node[markedv] (t13) at (9.2,0.35) {$13$};
\node[markedv] (t23) at (10.6,0.35) {$23$};
\foreach \a/\b in {t12/t13,t12/t23,t13/t23} {
  \draw[e] (\a) -- (\b);
}

\end{tikzpicture}
\end{adjustbox}
\caption{The two standard clique families in $J(5,2)$. A Johnson star fixes a common $(k-1)$-subset, while a Johnson top consists of all $k$-subsets of a common $(k+1)$-set.}
\label{fig:johnson-star-top}
\end{figure}

\section{Token Sliding}
\label{sec:TS}

\subsection{Some structural properties of \texorpdfstring{$\TS{k}{H}$}{TS\_k(H)}}
\label{sec:TS-structural}

We use the following structural properties of $\TS{k}{H}$ from \cite[Lem.~1 and Thm.~1]{LamPH26}.

\begin{lemma}[{\cite[Lem.~1]{LamPH26}}]\label{lem:clique-structure}
Let $H$ be a graph, let $k\ge 1$, and suppose that $\TS{k}{H}$ contains a clique $Q$ of size $q\ge 3$ with vertex set $A_1,\dots,A_q$.
Then one of the following holds.
\begin{enumerate}[(1)]
\item There exists a $(k+1)$-clique $U$ in $H$ such that each $A_i$ has the form $U-u_i$, where $u_1,\dots,u_q$ are pairwise distinct vertices of $U$.
In particular, $q\le k+1$.
\item There exists a $(k-1)$-clique $\Int$ in $H$ such that each $A_i$ has the form $\Int+a_i$, where $a_1,\dots,a_q\notin \Int$ are pairwise distinct.
\end{enumerate}
In particular, when $q>k+1$ only the second case can occur.
\end{lemma}
In the second case of \cref{lem:clique-structure}, the set $\Int\cup\{a_1,\dots,a_q\}$ is a clique in $H$: each $A_i$ is a clique, and the adjacency of $A_i$ and $A_j$ in $\TS{k}{H}$ gives $a_i a_j\in E(H)$ for every $i\ne j$.

\begin{lemma}[{\cite[Thm.~1]{LamPH26}}]\label{lem:clique-number-formula}
Let $H$ be a graph and let $k\ge 1$.
\begin{enumerate}[(1)]
\item If $k>\omega(H)$, then $\TS{k}{H}$ has no vertices.
\item If $k=\omega(H)$, then $\TS{k}{H}$ has no edges.
\item If $k<\omega(H)$, then $\omega(\TS{k}{H})=\max\{k+1,\,\omega(H)-k+1\}$.
\end{enumerate}
\end{lemma}

\subsection{Basic Graph Families and Their Complements}
\label{sec:TS-basic}

We first determine the exact sets $\KTS(G)$ for several basic graph families $G$.
\begin{theorem}\label{thm:warmup-ts}
The following exact formulas hold.
\begin{enumerate}[(1)]
\item For complete graphs,
\[
\KTS(K_n)=
\begin{cases}
\{k:k\ge 1\}, & n=1,\\
\{1\}, & n=2,\\
\{1,n-1\}, & n\ge 3.
\end{cases}
\]
\item For paths, cycles, and complete bipartite graphs,
\begin{align*}
\KTS(P_n) &= \{1\}\quad (n\ge 2),\\
\KTS(C_n) &= \{1\}\quad (n\ge 4),\\
\KTS(K_{m,n}) &= \{1\}\quad (m,n\ge 1).
\end{align*}
In particular, $\KTS(K_{1,n})=\{1\}$ for every $n\ge 1$.
\item For book graphs,
\[
\KTS(B_p)=
\begin{cases}
\{1,2\}, & p=1,\\
\{1\}, & p\ge 2.
\end{cases}
\]
\item For friendship graphs,
\[
\KTS(F_p)=\{1,2\}\qquad (p\ge 1).
\]
\end{enumerate}
\end{theorem}

\begin{proof}
\begin{enumerate}[(1)]
\item For complete graphs, $\TS{1}{K_n}\cong K_n$, so $1$ is always feasible.
For $n\ge 2$, also $\TS{n-1}{K_n}\cong J(n,n-1)\cong K_n$.
If $n=1$, then $\TS{k}{K_k}\cong K_1$ for every $k\ge 1$.
Thus, the displayed TS values are feasible.

Conversely, suppose there is a graph $H$ such that $K_n\cong \TS{k}{H}$ with $k\ge 2$, $n\ge 2$, and $n\neq k+1$.
Since $K_n$ has an edge, \cref{lem:clique-number-formula} gives $n=\omega(K_n)=\omega(\TS{k}{H})\ge k+1$.
Because $n\neq k+1$, we have $n>k+1$.
Applying \cref{lem:clique-structure} to the $n$ pairwise adjacent vertices of $K_n$, only the second case can occur.
Hence, there exists a $(k-1)$-clique $\Int$ and pairwise distinct vertices $a_1,\dots,a_n$ such that all vertices of $K_n$ are the $k$-cliques $\Int+a_i$, and $C:=\Int\cup\{a_1,\dots,a_n\}$ is a clique of size $k-1+n$ in $H$.
Every $k$-subset of $C$ is therefore a $k$-clique of $H$, so
\[
|V(\TS{k}{H})|\ge \binom{k-1+n}{k}>n,
\]
a contradiction.
This argument proves the complete-graph formula.

\item For paths, cycles, and complete bipartite graphs, each target has at least one edge and clique number at most $2$.
Thus, if $k\ge 2$ and $G\cong \TS{k}{H}$, then \cref{lem:clique-number-formula} gives $\omega(G)=\omega(\TS{k}{H})\ge k+1\ge 3$, impossible.
So only $k=1$ remains, and $\TS{1}{G}\cong G$ for every graph $G$.
This argument proves the formulas for $P_n$, $C_n$, and $K_{m,n}$, including stars $K_{1,n}$.

\item For book graphs, the case $p=1$ is $B_1=K_3$, already covered.
Assume now that $p\ge 2$.
The value $1$ is feasible because $\TS{1}{B_p}\cong B_p$.
Since $\omega(B_p)=3$, the same clique-number obstruction excludes every $k\ge 3$.
It remains to exclude $k=2$.
Suppose $B_p\cong \TS{2}{H}$.
Write the vertices of $B_p$ as $u,v,w_1,\dots,w_p$, where $u$ and $v$ are the two vertices of degree $p+1$ and each $w_i$ has degree $2$.
For any edge $e$ of $H$, each triangle containing $e$ contributes exactly two neighbors of $e$ in $\TS{2}{H}$, so the edge corresponding to each $w_i$ lies in exactly one triangle of $H$.
Since $w_i$ is adjacent to both $u$ and $v$, that unique triangle must contain the two edges corresponding to $u$ and $v$.
But two adjacent edges of a graph determine at most one triangle, so all $w_i$ would correspond to the same third edge of that triangle, a contradiction.
Hence, $2\notin\KTS(B_p)$ for $p\ge 2$.

\item For friendship graphs, the case $p=1$ is again $K_3$.
Assume now that $p\ge 2$.
The value $1$ is feasible because $\TS{1}{F_p}\cong F_p$.
Also $2$ is feasible because $\TS{2}{B_p}\cong F_p$: the vertices of $\TS{2}{B_p}$ are the edges of $B_p$, and inside each triangle of the book the three edges form a triangle in $\TS{2}{B_p}$, while edges from different pages do not create extra adjacencies.
Thus, the $p$ page-triangles of $B_p$ become $p$ triangles sharing the common vertex corresponding to the edge $uv$ of $B_p$, which is exactly the friendship graph $F_p$.
Finally, since $\omega(F_p)=3$, the clique-number obstruction excludes every $k\ge 3$.
Therefore, $\KTS(F_p)=\{1,2\}$ for all $p\ge 1$.
\end{enumerate}
\end{proof}

We also record the corresponding complement families of the above basic graph families.
Before going to the main results, we need the following observations.

The next proposition handles disjoint unions of cliques.
\begin{proposition}\label{prop:camera-complement-clique-unions-TS}
Let
\[
G=\bigsqcup_{i=1}^t K_{n_i}
\qquad (t\ge 1,\ n_i\ge 1).
\]
Then
\[
\KTS(G)=\{1\}\cup \{\,k\ge 2: n_i\in\{1,k+1\}\text{ for every }i\,\}.
\]
\end{proposition}

\begin{proof}
The case $k=1$ is immediate.
Now fix $k\ge 2$.
If every $n_i$ belongs to $\{1,k+1\}$, take the disjoint union of one copy of $K_k$ for each index with $n_i=1$ and one copy of $K_{k+1}$ for each index with $n_i=k+1$.
Since every $k$-clique of a disjoint union lies in a single component, the graph $\sfTS_k$ of that disjoint union is the disjoint union of the $\sfTS_k$-graphs of its components.
By the complete-graph formula in \cref{thm:warmup-ts}, this graph is exactly $G$.

Conversely, suppose $\TS{k}{H}\cong G$ for some graph $H$.
Let $X$ be a connected component of $\TS{k}{H}$.
If $|V(X)|=1$, there is nothing to prove.
Otherwise $X$ contains an edge $AB$.
Then $A\cup B$ is a $(k+1)$-clique of $H$, so all $k+1$ $k$-subsets of $A\cup B$ lie in the same connected component as $A$ and $B$.
Hence, $|V(X)|\ge k+1\ge 3$, and since $X$ is complete, \cref{lem:clique-structure} applies to all vertices of $X$.
If the first case of \cref{lem:clique-structure} holds, let $U$ be the corresponding $(k+1)$-clique.
Since $X$ is a connected component of $\TS{k}{H}$, it is closed under adjacency in $\TS{k}{H}$.
Thus, every $k$-subset of $U$ lies in $X$: any such subset not already listed is adjacent in $\TS{k}{H}$ to a vertex of $X$ of the form $U-u_i$.
The first case also says that every vertex of $X$ is one of these $k$-subsets, so $|V(X)|=k+1$.
If the second case holds, then there exist a $(k-1)$-clique $S$ and pairwise distinct vertices $a_1,\dots,a_q$ such that the vertices of $X$ are the cliques
\[
Q_i=S+a_i
\qquad (1\le i\le q),
\]
and $S\cup\{a_1,\dots,a_q\}$ is a clique of $H$.
Choosing $s\in S$ and distinct $i,j$, the $k$-clique $(S-s)+a_i+a_j$ also lies in that clique.
It is adjacent in $\TS{k}{H}$ to both $Q_i$ and $Q_j$, hence belongs to the same connected component, but it is different from every $Q_t$.
This contradiction shows that $X$ was not a full connected component.
Therefore, every connected component of $\TS{k}{H}$ has size either $1$ or $k+1$, which proves the displayed formula.
\end{proof}

The next lemma gives a divisibility condition on the degree in $\sfTS_k$ graphs.
\begin{lemma}\label{lem:camera-ts-degree-divisibility}
Let $H$ be a graph and let $k\ge 1$.
Let $A$ be a vertex of $\TS{k}{H}$.
Then
\[
\deg_{\TS{k}{H}}(A)=k\cdot t_A,
\]
where $t_A$ is the number of $(k+1)$-cliques of $H$ containing $A$.
In particular, if $G\cong \TS{k}{H}$ with $k\ge 2$, then every vertex degree of $G$ is divisible by $k$.
\end{lemma}

\begin{proof}
Every neighbor $B$ of $A$ differs from $A$ by replacing one vertex of $A$ with one vertex outside $A$.
Hence, $A\cup B$ is a $(k+1)$-clique of $H$ containing $A$.
Conversely, if $C$ is a $(k+1)$-clique containing $A$, then for each $a\in A$ the set $C-a$ is a neighbor of $A$ in $\TS{k}{H}$.
Different choices of $C$ and $a$ give different neighbors.
Thus, each $(k+1)$-clique containing $A$ contributes exactly $k$ neighbors of $A$, proving the formula.
The divisibility statement is immediate.
\end{proof}

We are now ready to determine $\KTS(G)$ for the complements of the basic graph families in \cref{thm:warmup-ts}.
\begin{theorem}\label{thm:camera-ts-basic-complements}
The following formulas hold.
\begin{enumerate}[(1)]
\item
\[
\KTS(\overline{K_n})=\{k:k\ge 1\}
\qquad (n\ge 1).
\]
\item
\[
\KTS(\overline{K_{1,n}})=
\begin{cases}
\{k:k\ge 1\}, & n=1,\\
\{1,n-1\}, & n\ge 2.
\end{cases}
\]
\item For $2\le m\le n$,
\[
\KTS(\overline{K_{m,n}})=
\begin{cases}
\{1,n-1\}, & m=n,\\
\{1\}, & m<n.
\end{cases}
\]
\item
\[
\KTS(\overline{B_p})=
\begin{cases}
\{k:k\ge 1\}, & p=1,\\
\{1,p-1\}, & p\ge 2.
\end{cases}
\]
\item
\[
\KTS(\overline{P_n})=
\begin{cases}
\{k:k\ge 1\}, & n=2,\\
\{1\}, & n\ge 3.
\end{cases}
\]
\item
\[
\KTS(\overline{C_n})=\{1\}
\qquad (n\ge 4).
\]
\item
\[
\KTS(\overline{F_p})=
\begin{cases}
\{k:k\ge 1\}, & p=1,\\
\{1\}, & p=2,\\
\{1,2\}, & p=3,\\
\{1\}, & p\ge 4.
\end{cases}
\]
\end{enumerate}
\end{theorem}

\begin{proof}
\begin{enumerate}[(1)]
	\item The formula follows from \cref{prop:camera-complement-clique-unions-TS} and the decomposition $\overline{K_n} = nK_1$.
	\item The formula follows from \cref{prop:camera-complement-clique-unions-TS} and the decomposition $\overline{K_{1,n}} = K_n\sqcup K_1$.
	\item The formula follows from \cref{prop:camera-complement-clique-unions-TS} and the decomposition $\overline{K_{m,n}} = K_m\sqcup K_n$.
	\item The formula follows from \cref{prop:camera-complement-clique-unions-TS} and the decomposition $\overline{B_p} = K_p\sqcup 2K_1$.
	\item For paths, the case $n=2$ follows from $\overline{P_2}=2K_1=\overline{K_2}$ and (1).
	Now let $n\ge 3$ and suppose for some graph $H$ that $\overline{P_n}\cong \TS{k}{H}$ with $k\ge 2$.
	In $\overline{P_n}$, the two endpoints of $P_n$ have degree $n-2$, and the remaining $n-2$ vertices have degree $n-3$.
	By \cref{lem:camera-ts-degree-divisibility}, the integer $k$ divides both $n-2$ and $n-3$, hence divides $1$, impossible.
	Therefore, only $k=1$ is feasible, so $\KTS(\overline{P_n})=\{1\}$.
	\item For cycles, the cases $n=4$ and $n=5$ are $\overline{C_4}=2K_2$ and $\overline{C_5}=C_5$, respectively.
	Thus, the formula follows from \cref{prop:camera-complement-clique-unions-TS,thm:warmup-ts}.
	Now let $n\ge 6$ and suppose $\overline{C_n}\cong \TS{j}{H}$ for some $j\ge 2$.
		Fix a vertex $x$ of $\overline{C_n}$, and let $Q=\{a_1,\dots,a_j\}$ be the corresponding $j$-clique of $H$.
		Set
		\[
		S:=\bigcap_{i=1}^j N_H(a_i),\qquad m:=|S|.
		\]
		For each $i\in[j]$, let
		\[
		L_i:=\{Q-a_i+s:s\in S\}.
		\]
		Every neighbor of $Q$ in $\TS{j}{H}$ is obtained uniquely by replacing one $a_i$ with one vertex of $S$, so $N_{\overline{C_n}}(x)$ is partitioned into the $j$ layers $L_1,\dots,L_j$, each of size $m$.
		If $i\ne h$, then the vertices $Q-a_i+s\in L_i$ and $Q-a_h+t\in L_h$ are adjacent if and only if $s=t$.
		Indeed, for $s\ne t$ they differ in two vertices from each side, while for $s=t$ the exchanged vertices $a_i$ and $a_h$ are adjacent because $Q$ is a clique.
		Thus, a fixed neighborhood vertex $Q-a_i+s$ is nonadjacent to every vertex $Q-a_h+t$ with $h\ne i$ and $t\ne s$.
		Therefore, each neighborhood vertex has at least $(j-1)(m-1)$ nonneighbors inside the neighborhood, while
	\[
	jm=|N_{\overline{C_n}}(x)|=n-3.
	\]
	But $N_{\overline{C_n}}(x)\cong \overline{P_{n-3}}$, and every vertex of $\overline{P_{n-3}}$ has at most two nonneighbors inside that neighborhood.
	Thus,
	\[
	(j-1)(m-1)\le 2.
	\]
	If $n\in\{6,8\}$, then $n-3\in\{3,5\}$, so the factorization $jm=n-3$ with $j\ge 2$ forces $m=1$, making the neighborhood complete, contradiction.
	If $n=7$, then $jm=4$. If $m=1$, then the neighborhood is complete, impossible. Hence, $(j,m)=(2,2)$; but then the neighborhood is either $2K_2$ or $C_4$, never $\overline{P_4}=P_4$.
	If $n\ge 9$, then $jm\ge 6$. If $m=1$, then the neighborhood is complete, impossible. Hence, $m\ge 2$, and now $(j-1)(m-1)\le 2$ forces $(j,m)\in\{(2,3),(3,2)\}$.
	In both remaining cases every neighborhood vertex has at least two nonneighbors inside the neighborhood, whereas $\overline{P_{n-3}}$ has vertices with only one such nonneighbor.
	Hence, $\KTS(\overline{C_n})=\{1\}$ for all $n\ge 4$.

	\item For friendship complements, note that
	\[
	\overline{F_p}=K_1\sqcup \CP{p}.
	\]
	For $p=1$, we have $\overline{F_1}=3K_1$, so the first formula gives $\KTS(\overline{F_1})=\{k:k\ge 1\}$.
	For $p=2$, we have $\overline{F_2}=K_1\sqcup C_4$.
	Suppose $\overline{F_2}\cong \TS{k}{H}$ with $k\ge 2$.
	Every edge of a graph of the form $\TS{k}{H}$ lies in a triangle, because if two vertices $A$ and $B$ are adjacent, then $A\cup B$ is a $(k+1)$-clique and all $k$-subsets of $A\cup B$ form a clique of size $k+1\ge 3$ in $\TS{k}{H}$.
	However, the four edges of the $C_4$-component of $\overline{F_2}$ lie in no triangle.
	Therefore, $k\ge 2$ is impossible, and $\KTS(\overline{F_2})=\{1\}$.

	For $p=3$, the value $1$ is feasible because $\TS{1}{\overline{F_3}}\cong \overline{F_3}$.
	Also $2$ is feasible: if $H=K_2\sqcup K_4$, then the unique edge of the $K_2$-component becomes an isolated vertex of $\TS{2}{H}$, while every two adjacent edges of the $K_4$-component lie in a triangle and hence remain adjacent in $\TS{2}{H}$.
	Thus,
	\[
	\TS{2}{H}\cong K_1\sqcup L(K_4)=K_1\sqcup \CP{3}=\overline{F_3}.
	\]
	Now suppose $\overline{F_3}\cong \TS{k}{H}$ with $k\ge 3$.
	Every nonisolated vertex of $\overline{F_3}$ has degree $4$, so \cref{lem:camera-ts-degree-divisibility} implies that $k$ divides $4$.
	Hence, $k=4$.
		Fix a nonisolated vertex $x$ of $\overline{F_3}$, and let $Q=\{a_1,a_2,a_3,a_4\}$ be the corresponding $4$-clique of $H$.
		Set
		\[
		S:=\bigcap_{i=1}^4 N_H(a_i),\qquad m:=|S|.
		\]
		Every neighbor of $Q$ in $\TS{4}{H}$ is obtained uniquely by replacing one vertex of $Q$ with one vertex of $S$, so
		\[
		4m=\deg_{\overline{F_3}}(x)=4,
		\]
		Thus, $m=1$.
		Then the four neighbors of $x$ use the same replacement vertex, and any two of them are adjacent because the two exchanged vertices of $Q$ are adjacent.
		Hence, $N_{\overline{F_3}}(x)$ is complete, contrary to the fact that every nonisolated vertex of $K_1\sqcup \CP{3}$ has neighborhood isomorphic to $C_4$.
	Therefore, $k\ge 3$ is impossible, and $\KTS(\overline{F_3})=\{1,2\}$.

		Finally, let $p\ge 4$ and suppose $\overline{F_p}\cong \TS{k}{H}$ with $k\ge 2$.
		Fix a nonisolated vertex $x$ of $\overline{F_p}$, and let $Q=\{a_1,\dots,a_k\}$ be the corresponding $k$-clique of $H$.
		Set
		\[
		S:=\bigcap_{i=1}^k N_H(a_i),\qquad m:=|S|.
		\]
		Every neighbor of $Q$ in $\TS{k}{H}$ is obtained uniquely by replacing one vertex of $Q$ with one vertex of $S$, so
		\[
		km=\deg_{\overline{F_p}}(x)=2p-2.
		\]
	Also every nonisolated vertex of $\CP{p}$ has neighborhood isomorphic to $\CP{p-1}$, and each vertex of $\CP{p-1}$ has exactly one nonneighbor inside that neighborhood.
	If $m=1$, then the neighborhood of $x$ is complete, contradiction.
	Therefore, $m\ge 2$.
		As in the cycle case, each neighborhood vertex has at least $(k-1)(m-1)$ nonneighbors inside the neighborhood, so
		\[
		(k-1)(m-1)\le 1.
		\]
	Since $k\ge 2$ and $m\ge 2$, this forces $k=m=2$.
	Then $km=4$, so $2p-2=4$ and hence $p=3$, contradiction.
	Therefore, no value $k\ge 2$ is feasible for $p\ge 4$, and so $\KTS(\overline{F_p})=\{1\}$.
	This completes the proof.
\end{enumerate}
\end{proof}

\subsection{Johnson Graphs}
\label{sec:TS-Johnson}

We now turn to Johnson graphs on the TS side.
The main result of this section determines $\KTS(J(n,r))$ for every nontrivial Johnson graph $J(n,r)$.
The section is organized by the cases appearing in the main theorem: first the complete-graph levels $r=1$ and $r=n-1$, then the stable higher-rank case $n>2r\ge 4$, and finally the boundary case $n=2r\ge 4$.

The arguments in the two nontrivial regimes use different templates.
In the stable range $n>2r$, we first identify the Johnson maximum cliques as Johnson stars and record how many of them pass through a vertex or an edge; these data are then compared with the divisibility and local-count constraints coming from a hypothetical $\sfTS_j$ witness.
In the boundary range $n=2r$, the method changes: instead of Johnson-star counting, we compare the rook-graph neighborhood $L(K_{r,r})$ with the layered neighborhood structure forced by $\sfTS_j$ graphs.
Here, the \emph{$m \times n$ rook graph} is the line graph of the complete bipartite graph $K_{m,n}$, and can be identified with the graph on the cells of an $m\times n$ grid, where two cells are adjacent exactly when they lie in a common row or in a common column (e.g., like the moves of a rook in a chessboard).

\begin{theorem}\label{thm:camera-ts-johnson}
For every pair of integers $n\ge 2$ and $1\le r\le n-1$, let
\[
s:=\min\{r,n-r\}.
\]
Then
\[
\KTS(J(n,r))=
\begin{cases}
\{1\}, & n=2,\\
\{1,n-1\}, & n\ge 3\text{ and }s=1,\\
\{1,s\}, & s\ge 2\text{ and }n=2s,\\
\{1,s,n-s\}, & s\ge 2\text{ and }n>2s.
\end{cases}
\]
\end{theorem}

The proof is deferred until the end of the subsection, after the complete-graph, stable, and boundary cases have been established.

We start the Johnson analysis with the standard neighborhood description, stated here in line-graph form and illustrated in \cref{fig:jnr-neighborhood-line}; see, for example, Brouwer, Cohen, and Neumaier~\cite[p.~256]{BrouwerCN89}.
\begin{lemma}\label{lem:camera-jnr-neighborhood-line}
Let $1\le r\le n-1$, and let $A$ be a vertex of $J(n,r)$.
Then
\[
N_{J(n,r)}(A)\cong L(K_{r,n-r}).
\]
\end{lemma}

\begin{figure}[ht]
\centering
\begin{adjustbox}{max width=\textwidth}
\begin{tikzpicture}[x=1.18cm,y=1.12cm]
\tikzset{
  title/.style={font=\bfseries\small, align=center},
  note/.style={font=\small, align=center},
  tiny/.style={font=\scriptsize, align=center},
  gridline/.style={draw=black!55, line width=.45pt},
  emphcell/.style={fill=black!10, draw=black, line width=.75pt},
  v/.style={circle, draw=black!70, fill=white, minimum size=21pt, inner sep=1pt, font=\scriptsize},
  emphv/.style={v, draw=black, fill=black!8, line width=.75pt},
  e/.style={draw=black!45, line width=.45pt},
  emphe/.style={draw=black, line width=1.1pt},
  maparrow/.style={->, draw=black!70, line width=.55pt}
}

\node[title] at (2.35,3.85) {neighbors of $A$ in $J(n,r)$};
\node[note] at (2.35,3.35) {$A=\{a_1,a_2,a_3\}$,\quad $B=\{b_1,b_2,b_3\}$};

\node[tiny] at (2.35,2.85) {add $b_j$};
\foreach \j/\x in {1/1.35,2/2.35,3/3.35} {
  \node[tiny] at (\x,2.45) {$b_\j$};
}
\foreach \i/\y in {1/1.95,2/1.25,3/0.55} {
  \node[tiny] at (0.45,\y) {$a_\i$};
}
\node[tiny, rotate=90] at (0.05,1.25) {delete $a_i$};

\foreach \x in {0.85,1.85,2.85,3.85} {
  \draw[gridline] (\x,0.2) -- (\x,2.3);
}
\foreach \y in {0.2,0.9,1.6,2.3} {
  \draw[gridline] (0.85,\y) -- (3.85,\y);
}
\path[emphcell] (1.85,0.9) rectangle (2.85,1.6);
\node[tiny] at (1.35,1.95) {$11$};
\node[tiny] at (2.35,1.95) {$12$};
\node[tiny] at (3.35,1.95) {$13$};
\node[tiny] at (1.35,1.25) {$21$};
\node[tiny] at (2.35,1.25) {$22$};
\node[tiny] at (3.35,1.25) {$23$};
\node[tiny] at (1.35,0.55) {$31$};
\node[tiny] at (2.35,0.55) {$32$};
\node[tiny] at (3.35,0.55) {$33$};

\draw[maparrow] (4.25,1.25) -- (6.25,1.25)
  node[midway, above=5pt, note, text width=3.2cm]
    {swap-neighbor $A-a_i+b_j$\\corresponds to edge $a_ib_j$};

\node[title] at (9.15,3.85) {$K_{r,n-r}$ edge model};
\node[v] (a1) at (7.25,2.60) {$a_1$};
\node[emphv] (a2) at (7.25,1.60) {$a_2$};
\node[v] (a3) at (7.25,0.60) {$a_3$};
\node[v] (b1) at (11.05,2.60) {$b_1$};
\node[emphv] (b2) at (11.05,1.60) {$b_2$};
\node[v] (b3) at (11.05,0.60) {$b_3$};
\foreach \i in {1,2,3} {
  \foreach \j in {1,2,3} {
    \draw[e] (a\i) -- (b\j);
  }
}
\draw[emphe] (a2) -- (b2);
\node[note] at (9.15,0.00) {vertices of $L(K_{r,n-r})$ are these edges};

\end{tikzpicture}
\end{adjustbox}
\caption{The neighborhood model for a vertex $A$ of $J(n,r)$. A neighbor is obtained by deleting some $a_i\in A$ and adding some $b_j\in [n]\setminus A$; the swap $A-a_i+b_j$ corresponds to the edge $a_ib_j$ of $K_{r,n-r}$. Adjacency between neighbors is therefore incidence between the corresponding bipartite edges.}
\label{fig:jnr-neighborhood-line}
\end{figure}

The next lemma identifies the neighborhood structure that arises from a $\sfTS_2$ witness.
\begin{lemma}\label{lem:camera-TS2-neighborhood-prism}
Let $H$ be a graph, let $G$ be a graph, and fix an isomorphism $\phi:G\to\TS{2}{H}$.
Let $e=\{u,v\}\in E(H)$, and put $x_e:=\phi^{-1}(e)$.
Put
\[
S:=N_H(u)\cap N_H(v).
\]
Then, after identifying each vertex $y$ of $N_G(x_e)$ with the edge $\phi(y)$ of $H$, the neighborhood $N_G(x_e)$ has vertex set
\[
\bigl\{\{u,s\}:s\in S\bigr\}\;\cup\;\bigl\{\{v,s\}:s\in S\bigr\},
\]
and is obtained from two copies of the induced graph $H[S]$ by adding the perfect matching
\[
\bigl\{\{u,s\},\{v,s\}\bigr\}:s\in S.
\]
In particular, every edge of this matching lies in no triangle of $N_G(x_e)$.
\end{lemma}

\begin{proof}
A vertex of $N_G(x_e)$ corresponds under $\phi$ to an edge of $H$ that is $\sfTS_2$-adjacent to $e=\{u,v\}$, hence exactly an edge of the form $\{u,s\}$ or $\{v,s\}$ with $s\in S=N_H(u)\cap N_H(v)$.
Thus, the displayed set is the full vertex set of $N_G(x_e)$.
Now let $s,t\in S$.
The vertices $\{u,s\}$ and $\{u,t\}$ are adjacent in $N_G(x_e)$ if and only if $st\in E(H)$, and the same holds for $\{v,s\}$ and $\{v,t\}$.
A cross pair $\{u,s\},\{v,t\}$ is adjacent if and only if $s=t$.
Therefore, $N_G(x_e)$ is precisely the graph described in the statement.
For the last claim, fix $s\in S$.
Any neighbor of $\{u,s\}$ inside $N_G(x_e)$ is either $\{v,s\}$ or a vertex of the form $\{u,t\}$ with $st\in E(H)$, while any neighbor of $\{v,s\}$ is either $\{u,s\}$ or a vertex of the form $\{v,t\}$ with $st\in E(H)$.
These two neighbor sets are disjoint, so the matching edge $\{u,s\}\{v,s\}$ lies in no triangle of $N_G(x_e)$.
\end{proof}

The complete-graph Johnson cases are exactly the complete-graph witnesses.
\begin{corollary}\label{cor:camera-johnson-complete-levels-ts}
For every integer $n\ge 2$,
\[
\KTS(J(n,1))=\KTS(J(n,n-1))=
\begin{cases}
\{1\}, & n=2,\\
\{1,n-1\}, & n\ge 3.
\end{cases}
\]
\end{corollary}

\begin{proof}
Since
\[
J(n,1)\cong K_n\cong J(n,n-1),
\]
the claim is exactly the complete-graph formula from \cref{thm:warmup-ts}.
\end{proof}

We can now exclude higher-rank Johnson graphs from the class of $\sfTS_2$-graphs.
\begin{theorem}\label{thm:camera-johnson-no-TS2}
For integers $n,r$ with $3\le r\le n-3$,
\[
2\notin \KTS(J(n,r)).
\]
\end{theorem}

\begin{proof}
Suppose to the contrary that $J(n,r)\cong \TS{2}{H}$ for some graph $H$, where $3\le r\le n-3$.
Fix a vertex $X$ of $J(n,r)$.
By \cref{lem:camera-jnr-neighborhood-line},
\[
N_{J(n,r)}(X)\cong L(K_{r,n-r}).
\]
Because $r\ge 3$ and $n-r\ge 3$, every vertex of $K_{r,n-r}$ has degree at least $3$.
Hence, if two edges of $K_{r,n-r}$ are incident, then their common endpoint is incident with a third edge.
Therefore, every edge of the line graph $L(K_{r,n-r})$ lies in a triangle.
On the other hand, since $J(n,r)\cong \TS{2}{H}$, fix an isomorphism $\phi:J(n,r)\to\TS{2}{H}$ and write $\phi(X)=e=\{u,v\}\in E(H)$.
Set $S:=N_H(u)\cap N_H(v)$.
The graph $L(K_{r,n-r})$ has $r(n-r)>0$ vertices, so $N_{J(n,r)}(X)$ is nonempty; by the vertex-set description in \cref{lem:camera-TS2-neighborhood-prism}, this implies $S\ne\emptyset$.
Choose $s\in S$.
By \cref{lem:camera-TS2-neighborhood-prism}, the edge between $\phi^{-1}(\{u,s\})$ and $\phi^{-1}(\{v,s\})$ in $N_{J(n,r)}(X)$ lies in no triangle.
This contradiction proves that $2\notin \KTS(J(n,r))$.
\end{proof}

The next proposition gives a divisibility condition coming from maximum cliques of $\sfTS_j$.
\begin{proposition}\label{prop:camera-TSj-max-clique-divisibility}
Let $H$ be a graph, let $j\ge 2$, and let $G$ be a graph with a fixed isomorphism to $\TS{j}{H}$.
Write $q=\omega(G)$ and let $c_t(X)$ denote the number of $t$-cliques of a graph $X$.
If $q\ge j+2$, then
\[
\binom{q+j-1}{j-1}\mid c_q(G).
\]
More precisely,
\[
c_q(G)=\binom{q+j-1}{j-1}\,c_{q+j-1}(H).
\]
\end{proposition}

\begin{proof}
Use the fixed isomorphism to regard the vertices of $G$ as the $j$-cliques of $H$.
Fix a $q$-clique $Q$ of $G$.
Since $q>j+1$, case~(1) of \cref{lem:clique-structure} cannot occur.
Hence, there exist a $(j-1)$-clique $\Int$ of $H$ and pairwise distinct vertices $a_1,\dots,a_q\notin \Int$ such that the vertices of $Q$ are exactly the $j$-cliques
\[
\Int+a_1,\dots,\Int+a_q,
\]
and $\Int\cup\{a_1,\dots,a_q\}$ is a clique of size $q+j-1$ in $H$.
The core $\Int$ is uniquely determined by $Q$, because it is the intersection of all $j$-cliques represented by the vertices of $Q$.
Thus, each $q$-clique $Q$ determines a unique pair $(K,\Int)$ where $K$ is a $(q+j-1)$-clique of $H$ and $\Int\subseteq V(K)$ has size $j-1$.
Conversely, if $K$ is a $(q+j-1)$-clique of $H$ and $\Int\subseteq V(K)$ has size $j-1$, then the $q$ vertices
\[
\{\Int+a:a\in V(K)\setminus \Int\}
\]
form a $q$-clique of $\TS{j}{H}$.
Hence,
\[
c_q(G)=\binom{q+j-1}{j-1}\,c_{q+j-1}(H),
\]
as required.
\end{proof}

We next record the maximum-clique structure of stable Johnson graphs for later use.
This is a standard consequence of the star/top classification of Johnson cliques~\cite[Thm.~3.4 and Cor.~3.5]{ShuldinerO22}; once the maximum cliques are known to be Johnson stars, the next results turn that structure into divisibility and incidence obstructions for $\sfTS_j$ witnesses.
\begin{lemma}\label{lem:camera-JnR-max-stars-stable}
Assume $n>2r\ge 4$.
Then every maximum clique of $J(n,r)$ is a Johnson star clique of size $n-r+1$.
Moreover, every vertex of $J(n,r)$ lies in exactly $r$ maximum cliques.
\end{lemma}

Specializing this divisibility condition yields the stable Johnson obstruction.
\begin{corollary}\label{cor:camera-johnson-general-max-clique-divisibility}
Let $n>2r\ge 4$, and let $j$ satisfy $2\le j\le n-r-1$.
If $j\in \KTS(J(n,r))$, then
\[
\binom{n-r+j}{j-1}\mid \binom{n}{r-1}.
\]
\end{corollary}

\begin{proof}
Assume $J(n,r)\cong \TS{j}{H}$ for some graph $H$.
By \cref{lem:camera-JnR-max-stars-stable}, the Johnson clique number is
\[
q:=\omega(J(n,r))=n-r+1.
\]
Because $j\le n-r-1$, we have $q\ge j+2$, so \cref{prop:camera-TSj-max-clique-divisibility} applies.
Again by \cref{lem:camera-JnR-max-stars-stable}, every maximum clique of $J(n,r)$ is a Johnson star clique and is uniquely determined by an $(r-1)$-subset of $[n]$, so the number of maximum cliques of $J(n,r)$ is
\[
\binom{n}{r-1}.
\]
Applying \cref{prop:camera-TSj-max-clique-divisibility} gives the claimed divisibility.
\end{proof}

The next proposition gives a second divisibility condition based on the incidence of maximum cliques.
\begin{proposition}\label{prop:camera-TSj-max-clique-incidence-divisibility}
Let $H$ be a graph, let $j\ge 2$, and let $G$ be a graph with a fixed isomorphism to $\TS{j}{H}$.
Write $q=\omega(G)$.
Assume $q\ge j+2$, and suppose that every vertex of $G$ lies in exactly $t$ maximum cliques of $G$.
Then
\[
j\mid t.
\]
\end{proposition}

\begin{proof}
Use the fixed isomorphism to regard the vertices of $G$ as the $j$-cliques of $H$.
Fix a vertex $X$ of $G$, and let $Q$ be the corresponding $j$-clique of $H$.
Let $\lambda_Q$ denote the number of $(q+j-1)$-cliques of $H$ containing $Q$.
Because $q>j+1$, case~(1) of \cref{lem:clique-structure} cannot occur for maximum cliques of $G$.
Hence, every maximum clique of $G$ containing $X$ is obtained by choosing a $(q+j-1)$-clique $K$ of $H$ containing $Q$, choosing a $(j-1)$-subset $\Int\subset Q$, and then taking the $q$ vertices
\[
\{\Int+a:a\in V(K)\setminus \Int\}
\]
of $\TS{j}{H}$.
Conversely, every such choice produces a maximum clique containing $X$.
Thus, the number of maximum cliques of $G$ containing $X$ is exactly
\[
\binom{j}{j-1}\lambda_Q=j\lambda_Q.
\]
By hypothesis this number is $t$, so $t=j\lambda_Q$, and therefore $j\mid t$.
\end{proof}

We now specialize this incidence condition to stable Johnson graphs.
\begin{corollary}\label{cor:camera-johnson-nondivisor-exclusion}
Let $n>2r\ge 4$, and let $j$ satisfy $2\le j\le r-1$.
If $j\in \KTS(J(n,r))$, then
\[
j\mid r.
\]
\end{corollary}

\begin{proof}
Assume $J(n,r)\cong \TS{j}{H}$ for some graph $H$.
By \cref{lem:camera-JnR-max-stars-stable}, the Johnson clique number is
\[
q:=\omega(J(n,r))=n-r+1,
\]
and every maximum clique is a Johnson star clique.
Each vertex of $J(n,r)$ lies in exactly $r$ such maximum cliques, one for each $(r-1)$-subset of the represented $r$-set.
Moreover,
\[
q=n-r+1\ge r+2>j+1,
\]
so \cref{prop:camera-TSj-max-clique-incidence-divisibility} applies with $t=r$.
Hence, $j\mid r$.
\end{proof}

Combining the previous obstructions yields the following stable reduction.
\begin{theorem}\label{thm:camera-johnson-stable-general-reduction}
Let $n>2r\ge 4$, and let
\[
j\in \KTS(J(n,r))\setminus\{1,r,n-r\}.
\]
Then
\[
3\le j\le r-1
\qquad\text{and}\qquad
j\mid r.
\]
\end{theorem}

\begin{proof}
Because $n>2r$, we have $\omega(J(n,r))=n-r+1$.
Choose $H$ with $J(n,r)\cong\TS{j}{H}$.
The graph $J(n,r)$ has edges, so the no-vertex and edgeless cases of \cref{lem:clique-number-formula} do not apply.
Hence, $j<\omega(H)$, and the TS clique-number formula gives $j\le n-r$.
Excluding the endpoint value $j=n-r$ yields
\[
2\le j\le n-r-1.
\]
Suppose first that $j\ge r+1$.
Then \cref{cor:camera-johnson-general-max-clique-divisibility} gives
\[
\binom{n-r+j}{j-1}\mid \binom{n}{r-1}.
\]
But
\[
\binom{n-r+j}{j-1}=\binom{n-r+j}{n-r+1}.
\]
Since $j\ge r+1$, we have $n-r+j\ge n+1$, so monotonicity in the upper argument gives
\[
\binom{n-r+j}{n-r+1}\ge \binom{n+1}{n-r+1}=\binom{n+1}{r}.
\]
Moreover,
\[
\frac{\binom{n+1}{r}}{\binom{n}{r-1}}=\frac{n+1}{r}>1,
\]
so
\[
\binom{n-r+j}{j-1}>\binom{n}{r-1},
\]
contradicting the divisibility above.
Therefore,
\[
j\le r-1.
\]
Since $2\le j\le r-1$, the case $r=2$ is impossible.
Thus, in the remaining argument we may assume $r\ge 3$.
If $j=2$, then
\[
3\le r\le n-r\le n-3,
\]
so \cref{thm:camera-johnson-no-TS2} excludes that possibility.
Hence,
\[
3\le j\le r-1.
\]
Now \cref{cor:camera-johnson-nondivisor-exclusion} applies and yields $j\mid r$.
\end{proof}

The next lemma gives the local counts needed in the stable argument.
\begin{lemma}\label{lem:camera-JnR-stable-local-counts}
Let $n>2r\ge 4$.
\begin{enumerate}[(1)]
\item If $A$ and $B$ are adjacent vertices of $J(n,r)$, then $A$ and $B$ have exactly $n-2$ common neighbors.
\item If $A,B,C$ are three distinct vertices contained in one maximum clique of $J(n,r)$, then $A,B,C$ have exactly $n-r-2$ common neighbors.
\end{enumerate}
\end{lemma}

\begin{proof}
For (1), write
\[
A=S\cup\{a\},\qquad B=S\cup\{b\},\qquad |S|=r-1,
\]
with $a\ne b$ and $a,b\notin S$.
A common neighbor $X$ of $A$ and $B$ must satisfy $|X\cap A|=|X\cap B|=r-1$.
There are exactly two possibilities.
First, $X$ may contain all of $S$.
Then $X=S\cup\{d\}$ where
\[
d\in [n]\setminus(S\cup\{a,b\}),
\]
giving $n-r-1$ choices.
Second, $X$ may omit exactly one element of $S$; then to stay adjacent to both $A$ and $B$, it must contain both $a$ and $b$, so
\[
X=(S\setminus\{s\})\cup\{a,b\}
\]
for some $s\in S$, giving $r-1$ choices.
Thus, $A$ and $B$ have exactly
\[
(n-r-1)+(r-1)=n-2
\]
common neighbors.

For (2), by \cref{lem:camera-JnR-max-stars-stable}, every maximum clique of $J(n,r)$ is a Johnson star clique.
Hence, we may write
\[
A=S\cup\{a\},\qquad B=S\cup\{b\},\qquad C=S\cup\{c\},\qquad |S|=r-1,
\]
with $a,b,c$ pairwise distinct and outside $S$.
Let $X$ be a common neighbor of $A,B,C$.
If $X$ omitted some element $s\in S$, then adjacency to each of $A,B,C$ would force $X$ to contain all of $a,b,c$, giving at least $r+1$ elements, impossible.
Therefore, $S\subseteq X$, and then necessarily
\[
X=S\cup\{d\}
\]
for some $d\in [n]\setminus(S\cup\{a,b,c\})$.
Conversely every such $X$ is adjacent to all three vertices.
Hence, $A,B,C$ have exactly
\[
n-(r-1)-3=n-r-2
\]
common neighbors.
\end{proof}

We now exclude the remaining proper divisors in the stable range.
This final stable argument combines the same maximum-clique information with the local common-neighbor counts from \cref{lem:camera-JnR-stable-local-counts}.
\begin{theorem}\label{thm:camera-johnson-stable-proper-divisor-exclusion}
Let $n>2r\ge 4$, and let $j$ satisfy
\[
2\le j\le r-1
\qquad\text{and}\qquad
j\mid r.
\]
Then
\[
j\notin \KTS(J(n,r)).
\]
\end{theorem}

\begin{proof}
Suppose to the contrary that
\[
J(n,r)\cong \TS{j}{H}
\]
for some graph $H$, where $n>2r\ge 4$ and $j$ is a proper divisor of $r$.
Write
\[
q:=\omega(J(n,r))=n-r+1.
\]
Because $j\le r/2$ and $n>2r$, we have
\[
q=n-r+1>r+1>j+1.
\]
Let $\mathcal B$ be the family of all $(q+j-1)=(n-r+j)$-cliques of $H$.

Fix a vertex $Q$ of $\TS{j}{H}$, viewed as a $j$-clique of $H$, and let $\lambda_Q$ be the number of members of $\mathcal B$ that contain $Q$.
Because $q>j+1$, case~(1) of \cref{lem:clique-structure} cannot occur for maximum cliques.
Hence, every maximum clique containing $Q$ is obtained by choosing a member of $\mathcal B$ containing $Q$ together with a $(j-1)$-subset of $Q$, and conversely every such choice produces a maximum clique containing $Q$.
This parametrization does not overcount: for a second-type maximum clique, the core is the intersection of its represented $j$-cliques and the block is their union.
Therefore, the number of maximum cliques containing $Q$ is exactly
\[
\binom{j}{j-1}\lambda_Q=j\lambda_Q.
\]
By \cref{lem:camera-JnR-max-stars-stable}, every vertex of $J(n,r)$ lies in exactly $r$ maximum cliques.
Hence
\[
r=j\lambda_Q.
\]
Thus, every $j$-clique of $H$ lies in exactly
\[
\lambda:=\frac{r}{j}
\]
members of $\mathcal B$.

Next fix an edge $QQ'$ of $\TS{j}{H}$.
Then $Q$ and $Q'$ are adjacent $j$-cliques, so they share a $(j-1)$-clique and their union
\[
R:=Q\cup Q'
\]
is a $(j+1)$-clique of $H$.
Since $q>j+1$, case~(1) of \cref{lem:clique-structure} cannot occur for a maximum clique of $\TS{j}{H}$.
Therefore, every maximum clique of $\TS{j}{H}$ is of the second type in \cref{lem:clique-structure}.
A maximum clique of $\TS{j}{H}$ containing both $Q$ and $Q'$ must then use the common $(j-1)$-set as core.
Conversely, each member of $\mathcal B$ that contains $R$ with this common core gives one such maximum clique, and distinct members give distinct cliques because the block is recovered as the union of the represented $j$-cliques.
Hence, the number of such maximum cliques is exactly the number of members of $\mathcal B$ that contain $R$.
In $J(n,r)$, every edge lies in exactly one maximum clique, namely the unique Johnson star clique determined by the common $(r-1)$-subset of the two adjacent $r$-sets.
Since every $(j+1)$-clique of $H$ is the union of two adjacent $j$-subcliques, the preceding edge argument applies to an arbitrary $(j+1)$-clique.
Therefore, every $(j+1)$-clique of $H$ lies in exactly one member of $\mathcal B$.

Fix one block $B\in\mathcal B$.
We first determine the outside common neighbors of the larger subsets of $B$.

\begin{claim}\label{clm:camera-johnson-stable-proper-divisor-exclusion-outside-common-neighbors}
	Every $(j+2)$-subset of $B$ has no common neighbor outside $B$.
\end{claim}
\begin{proof}
Let
\[
R=T\cup\{a,b,c\}\subseteq B, \qquad |T|=j-1.
\]
Consider the three $j$-cliques
\[
T+a,\ T+b,\ T+c,
\]
which form three vertices of $\TS{j}{H}$.
Because $B$ is a clique, these three vertices lie in one maximum clique of $\TS{j}{H}$ with core $T$, hence they correspond under the isomorphism to three vertices of a maximum clique of $J(n,r)$.
By \cref{lem:camera-JnR-max-stars-stable}, every maximum clique of $J(n,r)$ is a Johnson star clique of size $n-r+1$, so there are distinct elements $\alpha,\beta,\gamma\in [n]\setminus S$ such that the corresponding Johnson vertices are $S\cup\{\alpha\}$, $S\cup\{\beta\}$, and $S\cup\{\gamma\}$, where $|S|=r-1$.
By \cref{lem:camera-JnR-stable-local-counts}, they have exactly $n-r-2$ common neighbors in $J(n,r)$.

In $\TS{j}{H}$, let $X$ be a common neighbor of $T+a$, $T+b$, and $T+c$.
We claim that $T\subseteq X$.
Otherwise, let $t\in T\setminus X$.
Since $X$ is adjacent to each of $T+a$, $T+b$, and $T+c$, the set $X$ must then contain $a$, $b$, and $c$.
If moreover $X$ contains a vertex outside $R$, then $X$ contains
\[
(T\setminus\{t\})\cup\{a,b,c\},
\]
which already has size $(j-2)+3=j+1$, impossible because $|X|=j$.
Hence, $X\subseteq R$, and therefore $X=(T\setminus\{t\})\cup\{a,b,c\}$, which again has size $j+1$, a contradiction.
Thus, $T\subseteq X$, and since $|X|=j$, we have $X=T+x$ for some vertex $x\notin T$.
Because $X$ is adjacent to each of $T+a$, $T+b$, and $T+c$, the vertex $x$ is adjacent in $H$ to every vertex of
\[
R=T\cup\{a,b,c\}.
\]
Conversely, every vertex $x$ adjacent to all vertices of $R$ yields the common neighbor $T+x$.
Since $B$ is a clique of size $n-r+j$, exactly
\[
|B\setminus R|=(n-r+j)-(j+2)=n-r-2
\]
such vertices lie inside $B$.
Comparing with the Johnson count above, there are no vertices outside $B$ adjacent to all $j+2$ vertices of $R$.
\cref{clm:camera-johnson-stable-proper-divisor-exclusion-outside-common-neighbors} follows.
\end{proof}

An immediate consequence of \cref{clm:camera-johnson-stable-proper-divisor-exclusion-outside-common-neighbors} is that every outside vertex $x\in V(H)\setminus B$ satisfies
\[
|N_H(x)\cap B|\le j+1.
\]
Indeed, if $x$ had at least $j+2$ neighbors in $B$, then some $(j+2)$-subset of $B$ would have $x$ as an outside common neighbor.

\begin{claim}\label{clm:camera-johnson-stable-proper-divisor-exclusion-outside-common-neighbors-2}
Every $(j+1)$-subset of $B$ has exactly $r-j$ common neighbors outside $B$.
\end{claim}
\begin{proof}
Now fix a $(j+1)$-subset $R\subseteq B$.
Pick an edge of $\TS{j}{H}$ corresponding to two different $j$-subsets of $R$.
By \cref{lem:camera-JnR-stable-local-counts}, the corresponding edge of $J(n,r)$ has exactly $n-2$ common neighbors.
Write the chosen edge as $(T+a)(T+b)$, where $R=T\cup\{a,b\}$ and $|T|=j-1$.
Let $X$ be a common neighbor of $T+a$ and $T+b$ in $\TS{j}{H}$.
If $X\subseteq R$, then $X$ is a $j$-subset of $R$ different from $T+a$ and $T+b$, and hence $X$ is one of the other $j-1$ $j$-subsets of $R$.
Assume next that $X\nsubseteq R$.
We claim that $T\subseteq X$.
Otherwise, let $t\in T\setminus X$.
Since $X$ is adjacent to both $T+a$ and $T+b$, the set $X$ must contain both $a$ and $b$.
As $X$ also contains some vertex outside $R$, it then contains
\[
(T\setminus\{t\})\cup\{a,b\}
\]
plus one additional vertex outside $R$, and therefore has size at least $(j-2)+2+1=j+1$, impossible.
Thus, $T\subseteq X$.
Since $X\nsubseteq R$ and $|X|=j$, we obtain $X=T+x$ for a unique vertex $x\notin R$.
Now $X$ is adjacent to both $T+a$ and $T+b$ if and only if $x$ is adjacent in $H$ to every vertex of
\[
R=T\cup\{a,b\}.
\]
Therefore, the common neighbors of $T+a$ and $T+b$ in $\TS{j}{H}$ are exactly the other $j-1$ $j$-subsets of $R$, together with one vertex $T+x$ for each common neighbor $x$ of all $j+1$ vertices of $R$ in $H$.
Therefore, this $(j+1)$-clique $R$ has exactly
\[
(n-2)-(j-1)=n-j-1
\]
common neighbors in $H$.
Since exactly
\[
|B\setminus R|=(n-r+j)-(j+1)=n-r-1
\]
of these lie inside $B$, the number of common neighbors of $R$ outside $B$ is
\[
(n-j-1)-(n-r-1)=r-j.
\]
\cref{clm:camera-johnson-stable-proper-divisor-exclusion-outside-common-neighbors-2} follows.
\end{proof}

Now fix a $j$-subset $Q\subseteq B$.
For each vertex $y\in B\setminus Q$, let
\[
R_y:=Q\cup\{y\}.
\]
Then the sets $R_y$ are exactly the $(j+1)$-subsets of $B$ containing $Q$, so there are precisely
\[
|B\setminus Q|=(n-r+j)-j=n-r
\]
of them.
By \cref{clm:camera-johnson-stable-proper-divisor-exclusion-outside-common-neighbors-2}, each $R_y$ has exactly $r-j$ common neighbors outside $B$.
Moreover, no outside vertex can be counted for two different sets $R_y$ and $R_z$ with $y\neq z$. Indeed, such a vertex would be adjacent to all vertices of
\[
Q\cup\{y,z\},
\]
which is a $(j+2)$-subset of $B$, contradicting \cref{clm:camera-johnson-stable-proper-divisor-exclusion-outside-common-neighbors}.
Therefore, the number of vertices outside $B$ adjacent to all $j$ vertices of $Q$ is at least
\[
(r-j)(n-r).
\]
On the other hand, the vertex of $\TS{j}{H}$ corresponding to $Q$ has degree $r(n-r)$ in $J(n,r)$, so by \cref{lem:camera-ts-degree-divisibility} the $j$-clique $Q$ lies in exactly
\[
\frac{r(n-r)}{j}=\lambda(n-r)
\]
$(j+1)$-cliques of $H$, equivalently it has exactly $\lambda(n-r)$ common neighbors in $H$.
Here the equivalence is bijective: a common neighbor $x$ gives the $(j+1)$-clique $Q+x$, and every $(j+1)$-clique containing $Q$ contributes its unique vertex outside $Q$.
Of these, exactly $n-r$ lie inside $B$, namely the vertices of $B\setminus Q$.
So the number of common neighbors of $Q$ outside $B$ is exactly
\[
\lambda(n-r)-(n-r)=(\lambda-1)(n-r).
\]
Comparing with the lower bound above gives
\[
(r-j)(n-r)\le (\lambda-1)(n-r).
\]
Since $n-r>0$ and $r=j\lambda$, this simplifies to
\[
j(\lambda-1)\le \lambda-1.
\]
But $j\ge 2$ and $\lambda=r/j\ge 2$, so $j(\lambda-1)>\lambda-1$, a contradiction.
Therefore,
\[
j\notin \KTS(J(n,r)).
\]
\end{proof}

This corollary yields the complete stable TS formula for Johnson graphs.
\begin{corollary}\label{cor:camera-johnson-stable-all}
For every pair of integers $n>2r\ge 4$,
\[
\KTS(J(n,r))=\{1,r,n-r\}.
\]
\end{corollary}

\begin{proof}
The inclusion $\{1,r,n-r\}\subseteq \KTS(J(n,r))$ follows from the trivial level $1$, from the complete-graph witness $\TS{r}{K_n}\cong J(n,r)$, and from Johnson duality $J(n,n-r)\cong J(n,r)$ together with the witness $\TS{n-r}{K_n}\cong J(n,n-r)$.
Conversely, let
\[
j\in \KTS(J(n,r))\setminus\{1,r,n-r\}.
\]
By \cref{thm:camera-johnson-stable-general-reduction}, $j$ is a proper divisor of $r$ with $3\le j\le r-1$.
This conclusion contradicts \cref{thm:camera-johnson-stable-proper-divisor-exclusion}.
Hence, no such $j$ exists.
\end{proof}

We next describe the neighborhood layers in $\sfTS_j$ graphs.
These lemmas mark the shift to the boundary method: the remaining argument will compare the layered neighborhoods of a $\sfTS_j$ witness with the rook-graph neighborhoods of $J(2r,r)$.
A concrete instance of these layers is shown in \cref{fig:ts-neighborhood-layers}.
\begin{lemma}\label{lem:camera-TSj-neighborhood-layers}
Let $H$ be a graph, let $j\ge 2$, and let $G$ be a graph with a fixed isomorphism $\phi:G\to\TS{j}{H}$.
Let $Q=\{a_1,\dots,a_j\}$ be a $j$-clique of $H$, and put $x_Q:=\phi^{-1}(Q)$.
Put
\[
S:=\bigcap_{i=1}^j N_H(a_i),\qquad X:=H[S].
\]
Then, under the fixed identification by $\phi$, $N_G(x_Q)$ is obtained from $j$ disjoint copies $X_1,\dots,X_j$ of $X$ by adding, for each $s\in S$, all edges of the transversal clique $s_1\cdots s_j$ on the copies of $s$.
\end{lemma}

\begin{proof}
Since two $j$-cliques are adjacent in $\TS{j}{H}$ exactly when one is obtained from the other by replacing one vertex with an adjacent non-member, every neighbor of $x_Q$ corresponds under $\phi$ to a $j$-clique obtained by replacing exactly one vertex of $Q$ by a vertex $s\in S$.
For each $i\in [j]$, let $X_i$ denote the family of $j$-cliques
\[
Q-a_i+s \qquad (s\in S),
\]
and write $s_i$ for the copy of $s$ in $X_i$.
Equivalently, $s_i$ is the vertex $Q-a_i+s$ of $N_G(x_Q)$.
Two vertices $s_i$ and $t_i$ in the same layer are adjacent exactly when $st\in E(X)$, so each $X_i$ induces a copy of $X$.
If $i\neq i'$, then $s_i$ and $t_{i'}$ are adjacent exactly when $s=t$.
Therefore, the only cross-layer edges are the edges of the transversal clique $s_1\cdots s_j$ for each $s\in S$.
\end{proof}

\begin{figure}[ht]
\centering
\begin{adjustbox}{max width=\textwidth}
\begin{tikzpicture}[x=1cm,y=1cm]
\tikzset{
  title/.style={font=\bfseries\small, align=center},
  note/.style={font=\small, align=center},
  tiny/.style={font=\scriptsize, align=center},
  setup/.style={font=\small, align=center, text width=10.2cm, inner sep=2pt},
  layerlabel/.style={font=\small, align=center, text width=1.85cm},
  v/.style={rectangle, draw=black, fill=white, rounded corners=1pt,
    minimum width=1.80cm, minimum height=20pt, inner sep=3pt, font=\scriptsize},
  layeredge/.style={draw=black, line width=.65pt},
  transedge/.style={draw=black, line width=.55pt, densely dashed},
  legendline/.style={draw=black, line width=.65pt}
}

\node[title] at (5.2,8.00) {a concrete $\mathsf{TS}_4$ neighborhood};
\node[setup] at (5.2,7.30)
  {$Q=\{a_1,a_2,a_3,a_4\}$, $S=\{s,t,u\}$, and $H[S]$ is the path $s-t-u$; every vertex of $S$ is adjacent to every vertex of $Q$.};

\foreach \name/\x/\a in {1/1.0/a_1,2/3.8/a_2,3/6.6/a_3,4/9.4/a_4} {
  \node[layerlabel] at (\x,6.30) {$X_{\name}$\\delete $\a$};
}

\node[v] (s1) at (1.0,4.70) {$Q-a_1+s$};
\node[v] (s2) at (3.8,4.70) {$Q-a_2+s$};
\node[v] (s3) at (6.6,4.70) {$Q-a_3+s$};
\node[v] (s4) at (9.4,4.70) {$Q-a_4+s$};
\node[v] (t1) at (1.0,2.30) {$Q-a_1+t$};
\node[v] (t2) at (3.8,2.30) {$Q-a_2+t$};
\node[v] (t3) at (6.6,2.30) {$Q-a_3+t$};
\node[v] (t4) at (9.4,2.30) {$Q-a_4+t$};
\node[v] (u1) at (1.0,-0.10) {$Q-a_1+u$};
\node[v] (u2) at (3.8,-0.10) {$Q-a_2+u$};
\node[v] (u3) at (6.6,-0.10) {$Q-a_3+u$};
\node[v] (u4) at (9.4,-0.10) {$Q-a_4+u$};

\foreach \i in {1,2,3,4} {
  \draw[layeredge] (s\i) -- (t\i);
  \draw[layeredge] (t\i) -- (u\i);
}

\foreach \r in {s,t,u} {
  \draw[transedge] (\r1) -- (\r2);
  \draw[transedge] (\r2) -- (\r3);
  \draw[transedge] (\r3) -- (\r4);
}
\draw[transedge] (s1.north) .. controls (2.75,5.45) and (4.85,5.45) .. (s3.north);
\draw[transedge] (s2.north) .. controls (5.55,5.45) and (7.65,5.45) .. (s4.north);
\draw[transedge] (s1.north) .. controls (3.10,5.78) and (7.30,5.78) .. (s4.north);
\draw[transedge] (t1.north) .. controls (2.75,3.05) and (4.85,3.05) .. (t3.north);
\draw[transedge] (t2.north) .. controls (5.55,3.05) and (7.65,3.05) .. (t4.north);
\draw[transedge] (t1.north) .. controls (3.10,3.38) and (7.30,3.38) .. (t4.north);
\draw[transedge] (u1.north) .. controls (2.75,0.65) and (4.85,0.65) .. (u3.north);
\draw[transedge] (u2.north) .. controls (5.55,0.65) and (7.65,0.65) .. (u4.north);
\draw[transedge] (u1.north) .. controls (3.10,0.98) and (7.30,0.98) .. (u4.north);

\draw[legendline] (2.20,-1.05) -- (2.95,-1.05);
\node[tiny, anchor=west] at (3.05,-1.05) {copied edge of $H[S]$};
\draw[transedge] (5.85,-1.05) -- (6.60,-1.05);
\node[tiny, anchor=west] at (6.70,-1.05) {edge in a transversal clique};

\end{tikzpicture}
\end{adjustbox}
\caption{A concrete instance of the neighborhood layers in \cref{lem:camera-TSj-neighborhood-layers}.  Here $j=4$, $Q=\{a_1,a_2,a_3,a_4\}$, $S=\{s,t,u\}$, and $H[S]$ is the path $s-t-u$.  The column $X_i$ consists of the neighbors obtained by deleting $a_i$ from $Q$ and adding one vertex of $S$.  Solid vertical edges are copied from $H[S]$, while the dashed edges show all three transversal $K_4$ cliques.}
\label{fig:ts-neighborhood-layers}
\end{figure}

The next lemma records the maximum cliques of the rook graph.
\begin{lemma}\label{lem:camera-rook-max-cliques}
For every integer $r\ge 2$, the rook graph $L(K_{r,r})$ has clique number $r$.
Its maximum cliques are exactly the $r$ row cliques and the $r$ column cliques.
In particular, it has exactly $2r$ maximum cliques, and every maximum clique intersects exactly $r$ others.
\end{lemma}

\begin{proof}
Identify $L(K_{r,r})$ with the $r\times r$ rook graph on the cells $(i,j)$, where two cells are adjacent exactly when they lie in a common row or in a common column; this is the standard line-graph/grid identification for $K_{m,n}$~\cite[p.~440]{BrouwerCN89}.
Every row and every column is a clique of size $r$, so $\omega(L(K_{r,r}))\ge r$.
Conversely, if a clique contains two vertices from different rows and different columns, then those two vertices are nonadjacent; hence every clique is contained either in one row or in one column, so its size is at most $r$.
Therefore, $\omega(L(K_{r,r}))=r$, and the maximum cliques are exactly the $r$ rows and the $r$ columns.
A fixed row intersects every column in one vertex and is disjoint from the other $r-1$ rows, so it intersects exactly $r$ other maximum cliques; the same holds symmetrically for each column.
\end{proof}

A concrete version of the comparison used in the next proof is shown in \cref{fig:rook-ts-boundary}.
The right-hand panel is a conditional derivation under a hypothetical $\mathsf{TS}_2$ realization, not another neighborhood drawing.
\begin{figure}[ht]
\centering
\begin{adjustbox}{max width=\textwidth}
\begin{tikzpicture}[x=1.16cm,y=1.14cm]
\tikzset{
  title/.style={font=\bfseries\small, align=center},
  note/.style={font=\small, align=center},
  tiny/.style={font=\footnotesize, align=center},
  rowclique/.style={draw=black, line width=.45pt, rounded corners=6pt},
  colclique/.style={draw=black, line width=.45pt, densely dashed, rounded corners=6pt},
  selectedrow/.style={draw=black, line width=1.15pt, rounded corners=6pt},
  nodept/.style={circle, draw=black, fill=white, minimum size=5.5pt, inner sep=0pt},
  hitpt/.style={circle, draw=black, fill=black, minimum size=6.2pt, inner sep=0pt},
  layerbox/.style={draw=black, line width=.55pt},
  slot/.style={rectangle, draw=black, fill=white, minimum width=.82cm,
    minimum height=.46cm, inner sep=2pt, font=\footnotesize, align=center},
  openslot/.style={slot, densely dotted},
  demand/.style={font=\footnotesize, align=center},
  arrow/.style={->, draw=black, line width=.55pt},
  block/.style={draw=black, line width=.95pt}
}
\path[use as bounding box] (-.25,-.18) rectangle (12.18,6.32);

\node[title] at (2.20,5.72) {(a) actual $N(A)\cong L(K_{3,3})$};

\foreach \j/\x in {1/1.05,2/2.15,3/3.25} {
  \draw[colclique] (\x-.34,1.35) rectangle (\x+.34,4.05);
  \node[tiny] at (\x,4.36) {$C_{\j}$};
}
\foreach \i/\y in {1/3.75,2/2.70,3/1.65} {
  \draw[rowclique] (.72,\y-.22) rectangle (3.58,\y+.22);
  \node[tiny, anchor=east] at (.45,\y) {$R_{\i}$};
}
\draw[selectedrow] (.70,2.46) rectangle (3.60,2.94);

\foreach \x in {1.05,2.15,3.25} {
  \node[nodept] at (\x,3.75) {};
  \node[hitpt] at (\x,2.70) {};
  \node[nodept] at (\x,1.65) {};
}

\node[note] at (2.20,.78) {$R_2$ meets $C_1,C_2,C_3$};

\draw[densely dashed, line width=.45pt] (4.65,.45) -- (4.65,5.75);

\node[title] at (8.55,5.78) {(b) derivation under $J(6,3)\cong\mathsf{TS}_2(H)$};
\node[tiny] at (8.55,5.05) {assume a $\mathsf{TS}_2$ realization};
\node[tiny] at (8.55,4.55) {then intersecting maximum $3$-cliques\\ lie in one layer};

\draw[layerbox] (5.75,2.44) rectangle (10.72,3.34);
\draw[layerbox] (5.75,1.16) rectangle (10.72,2.06);
\node[tiny, anchor=east] at (5.60,2.89) {$X_1$};
\node[tiny, anchor=east] at (5.60,1.61) {$X_2$};

\node[slot] (r2) at (6.76,2.89) {$R_2$};
\node[openslot] (p1) at (8.14,2.89) {};
\node[openslot] (p2) at (9.52,2.89) {};

\node[openslot] at (6.76,1.61) {};
\node[openslot] at (8.14,1.61) {};
\node[openslot] at (9.52,1.61) {};

\node[demand] at (8.55,3.88) {from (a): required same-layer partners\\ $C_1,C_2,C_3$};
\draw[block] (10.42,2.91) -- (10.68,3.19);
\draw[block] (10.42,3.19) -- (10.68,2.91);

\node[note] at (8.55,.58) {$3$ required $>$ $2$ available};

\end{tikzpicture}
\end{adjustbox}
\caption{A concrete form of the boundary obstruction for $r=3$ and $j=2$.  Left, in the actual local graph $N(A)\cong L(K_{3,3})$, the selected row clique $R_2$ intersects all three column cliques $C_1,C_2,C_3$.  Right, under a hypothetical realization $J(6,3)\cong\TS{2}{H}$, intersecting maximum $3$-cliques must lie in the same $\mathsf{TS}_2$ layer; since the six maximum cliques split as three per layer, only two same-layer positions remain beside $R_2$.  Thus, the three required same-layer partners from the rook side cannot all be realized.}
\label{fig:rook-ts-boundary}
\end{figure}

We can now exclude all non-endpoint boundary values at once.
\begin{theorem}\label{thm:camera-johnson-boundary-general-exclusion}
For every integer $r\ge 3$ and every integer $j$ with $2\le j\le r-1$,
\[
j\notin \KTS(J(2r,r)).
\]
\end{theorem}

\begin{proof}
Suppose to the contrary that $J(2r,r)\cong \TS{j}{H}$ for some graph $H$, where $2\le j\le r-1$.
Fix a vertex $A$ of $J(2r,r)$, let $Q$ be the corresponding $j$-clique of $H$, and put
\[
S:=\bigcap_{x\in Q}N_H(x),\qquad X:=H[S].
\]
By \cref{lem:camera-TSj-neighborhood-layers}, the neighborhood $N_{J(2r,r)}(A)$ is obtained from $j$ disjoint copies $X_1,\dots,X_j$ of $X$ by adding, for each $s\in S$, the transversal clique $s_1\cdots s_j$.
On the other hand, by \cref{lem:camera-jnr-neighborhood-line},
\[
N_{J(2r,r)}(A)\cong L(K_{r,r}).
\]
By \cref{lem:camera-rook-max-cliques}, this neighborhood has clique number $r$, exactly $2r$ maximum cliques, and each maximum clique intersects exactly $r$ others.

Any clique meeting two different layers has size at most $j$: if a clique contains vertices from two different layers, then every cross-layer adjacency forces them to be copies of the same base vertex $s\in S$, so at most one vertex from each layer may occur.
Since $j\le r-1$, no clique of cardinality $r=\omega(L(K_{r,r}))$ can meet two different layers.
Thus, every clique of size $r$ is contained in a single layer $X_i$, and conversely every $r$-clique of a layer is a maximum clique of the whole neighborhood.
Let $m$ be the number of $r$-cliques of $X$.
Then each layer contributes exactly $m$ maximum cliques, so the layered neighborhood has exactly $jm$ maximum cliques.
Comparison with \cref{lem:camera-rook-max-cliques} gives
\[
jm=2r.
\]
Fix one maximum clique $C$ inside a layer, say $C\subseteq X_1$.
Any maximum clique that intersects $C$ must also lie in $X_1$, because different layers are disjoint and no maximum clique can use two layers.
Hence, $C$ intersects at most $m-1$ other maximum cliques.
But in the rook graph every maximum clique intersects exactly $r$ others.
Therefore,
\[
r\le m-1=\frac{2r}{j}-1\le r-1,
\]
a contradiction.
This argument proves that $j\notin \KTS(J(2r,r))$ for every $2\le j\le r-1$.
\end{proof}

We therefore obtain the complete boundary TS formula.
\begin{corollary}\label{cor:camera-johnson-boundary-all}
For every integer $r\ge 2$,
\[
\KTS(J(2r,r))=\{1,r\}.
\]
\end{corollary}

\begin{proof}
If $r=2$, then $1\in \KTS(J(4,2))$ is trivial, and $2\in \KTS(J(4,2))$ because $\TS{2}{K_4}\cong J(4,2)$.
Moreover, $\omega(J(4,2))=3$.
Suppose that $k\ge 3$ and $J(4,2)\cong \TS{k}{H}$ for some graph $H$.
Since $J(4,2)$ has an edge, \cref{lem:clique-number-formula} implies that $k<\omega(H)$.
Therefore,
\[
\omega(J(4,2))=\omega(\TS{k}{H})\ge k+1\ge 4,
\]
a contradiction.
Hence, $\KTS(J(4,2))=\{1,2\}$.
Assume now that $r\ge 3$.
The inclusion $\{1,r\}\subseteq \KTS(J(2r,r))$ is immediate from the trivial level $1$ and the complete-graph witness.
Conversely, if
\[
k\in \KTS(J(2r,r))\setminus\{1,r\},
\]
then there is a graph $H$ such that $J(2r,r)\cong \TS{k}{H}$.
Since $J(2r,r)$ has edges, \cref{lem:clique-number-formula} implies that $k<\omega(H)$.
	Moreover, $\omega(J(2r,r))=r+1$: by \cref{lem:camera-jnr-neighborhood-line,lem:camera-rook-max-cliques}, every clique containing a fixed vertex has size at most $r+1$, and Johnson stars and tops attain this size.
	The same formula gives
\[
r+1=\omega(\TS{k}{H})=\max\{k+1,\omega(H)-k+1\}.
\]
Hence, $k\le r$.
Together with $k\ge 1$ and $k\notin\{1,r\}$, this yields $2\le k\le r-1$.
\Cref{thm:camera-johnson-boundary-general-exclusion} excludes every such value.
\end{proof}

We are now ready to prove \Cref{thm:camera-ts-johnson}.
\begin{proof}[Proof of \Cref{thm:camera-ts-johnson}]
If $s=1$, then $J(n,r)\cong J(n,1)$ by Johnson duality, so the claim follows from \cref{cor:camera-johnson-complete-levels-ts}.
Assume from now on that $s\ge 2$.
By Johnson duality, we may replace $r$ by $s$ and therefore assume $r=s\le n/2$.
If $n=2s$, then the claim is exactly \cref{cor:camera-johnson-boundary-all}.
If $n>2s$, then the claim is exactly \cref{cor:camera-johnson-stable-all}.
These cases cover all remaining possibilities.
\end{proof}

\section{Token Jumping}
\label{sec:TJ}

\subsection{Some structural properties of \texorpdfstring{$\TJ{k}{H}$}{TJ\_k(H)}}
\label{subsec:TJ-structural}

Let $H$ be a graph and let $k\ge 1$.
If $S$ is a $(k-1)$-clique of $H$, then we write
\[
\mathcal C_H(S):=\{K\subseteq V(H): K \text{ is a $k$-clique of } H \text{ and } S\subseteq K\}.
\]
If $U$ is a $(k+1)$-clique of $H$, then we write $\mathcal T_H(U):=\{U-u:u\in U\}$.
For every $(k-1)$-clique $S$ of $H$, the set $\mathcal C_H(S)$ induces a clique in $\TJ{k}{H}$; we call this clique a \emph{star clique}. For every $(k+1)$-clique $U$ of $H$, the set $\mathcal T_H(U)$ induces a clique in $\TJ{k}{H}$; we call this clique a \emph{top clique}.
These are the TJ analogues of Johnson star and Johnson top cliques.

We next record the structure of cliques in $\TJ{k}{H}$; this will be used throughout the token-jumping section.
\begin{theorem}\label{thm:camera-TJ-clique-structure}
Let $H$ be a graph, let $k\ge 1$, and let $Q$ be a clique of $\TJ{k}{H}$ of size $m\ge 3$, whose vertices are $k$-cliques $A_1,\ldots,A_m$ of $H$.
Then exactly one of the following holds:
\begin{enumerate}[(1)]
\item there exists a $(k+1)$-clique $U$ in $H$ such that each $A_i$ has the form $U-u_i$, where $u_1,\ldots,u_m$ are pairwise distinct vertices of $U$; or
\item there exists a $(k-1)$-clique $S$ in $H$ such that each $A_i$ has the form $S+x_i$, where $x_1,\ldots,x_m\in V(H)\setminus S$ are pairwise distinct.
\end{enumerate}
In case~(1), $\bigl|\bigcup_{i=1}^m A_i\bigr|=k+1$; in case~(2), $\bigl|\bigcap_{i=1}^m A_i\bigr|=k-1$.
\end{theorem}

\begin{proof}
Since $Q$ is a clique in $\TJ{k}{H}$, every two distinct members $A_i,A_j$ are adjacent, and hence
\[
|A_i\cap A_j|=k-1 \qquad (1\le i<j\le m).
\]
We prove the theorem by induction on $m$.
For $m=3$, let $B=A_1\cap A_2$ and $C=A_1\cap A_3$.
Then $|B|=|C|=k-1$ and $B,C\subseteq A_1$, so
\[
k=|A_1|\ge |B\cup C|=|B|+|C|-|B\cap C|=2k-2-|B\cap C|.
\]
Hence, $|B\cap C|\ge k-2$.
Set $S=A_1\cap A_2\cap A_3=B\cap C$.
If $|S|=k-1$, then there are pairwise distinct vertices $x_1,x_2,x_3$ with
\[
A_1=S+x_1,\qquad A_2=S+x_2,\qquad A_3=S+x_3,
\]
which is case~(2).
If $|S|=k-2$, then there are pairwise distinct vertices $u_1,u_2,u_3$ with
\[
A_1=S+u_2+u_3,\qquad A_2=S+u_1+u_3,\qquad A_3=S+u_1+u_2.
\]
Let $U=S+u_1+u_2+u_3$.
Every pair of vertices of $U$ lies together in one of $A_1,A_2,A_3$, so $U$ is a $(k+1)$-clique of $H$ and $A_i=U-u_i$ for each $i$.
Thus, the theorem holds when $m=3$.

Assume now that $m\ge 4$ and the statement holds for $m-1$.
The family $A_1,\ldots,A_{m-1}$ is either of top type or of star type.
Apply the base case to the triangle $A_1,A_2,A_m$.
Since $|A_1\cap A_2|=k-1$, either
\[
A_m=(A_1\cap A_2)+x
\]
for some vertex $x$, or
\[
A_m=(A_1\cup A_2)-x
\]
for some vertex $x$.
We consider the two induction types in turn.

Suppose first that $A_1,\ldots,A_{m-1}$ are of top type, say $A_i=U-u_i$ for a $(k+1)$-clique $U$ and pairwise distinct $u_1,\ldots,u_{m-1}\in U$.
Then $A_1\cap A_2 = U-\{u_1,u_2\}$ and $A_1\cup A_2 = U$.
If $A_m=(A_1\cap A_2)+x$, then $x\notin U$; otherwise $x$ is $u_1$ or $u_2$, giving $A_m=A_1$ or $A_2$.
Choose $u_3$ distinct from $u_1,u_2$.
Since $A_3=U-u_3$ and $x\notin U$, we get
\[
A_m\cap A_3 = U-\{u_1,u_2,u_3\},
\]
which has size $k-2$, contradicting the adjacency of $A_m$ and $A_3$.
Therefore, this subcase is impossible.
Hence, $A_m=(A_1\cup A_2)-x = U-x$ for some $x\in U$.
Because $A_m$ is distinct from $A_1,\ldots,A_{m-1}$, the vertex $x$ is different from $u_1,\ldots,u_{m-1}$.
Thus, the enlarged family remains of top type.

Suppose next that $A_1,\ldots,A_{m-1}$ are of star type, say $A_i=S+x_i$ for a $(k-1)$-clique $S$ and pairwise distinct $x_1,\ldots,x_{m-1}\in V(H)\setminus S$.
Then $A_1\cap A_2=S$ and $A_1\cup A_2=S+x_1+x_2$.
If $A_m=(A_1\cap A_2)+x$, then $A_m=S+x$ with $x\notin S$ and $x\notin\{x_1,\ldots,x_{m-1}\}$ because $A_m$ is distinct from the earlier members.
So the enlarged family remains of star type.
If instead $A_m=(A_1\cup A_2)-x$, then $x$ must lie in $S$; if $x=x_1$ or $x=x_2$, then $A_m$ would equal $A_2$ or $A_1$.
Choose $x_3$ distinct from $x_1,x_2$.
Since $A_3=S+x_3$, we obtain
\[
A_m\cap A_3 = (S-\{x\})\cup \bigl(\{x_1,x_2\}\cap\{x_3\}\bigr) = S-\{x\},
\]
which has size $k-2$, again contradicting adjacency.
Therefore, this subcase is impossible.

This induction completes the proof.
In case~(1), all members are $k$-subsets of the same $(k+1)$-clique $U$, so their union is $U$.
In case~(2), all members contain the same $(k-1)$-clique $S$.
Moreover, the two cases are mutually exclusive when $m\ge 3$: in case~(1), the total intersection has size $|U|-m=(k+1)-m\le k-2$, while in case~(2) the total intersection has size $k-1$.
\end{proof}

The next corollary describes maximal cliques in $\TJ{k}{H}$.
\begin{corollary}\label{cor:camera-TJ-maximal-cliques}
Let $H$ be a graph, let $k\ge 1$, and let $Q$ be a maximal clique of $\TJ{k}{H}$.
Then at least one of the following holds:
\begin{enumerate}[(1)]
\item $Q=\mathcal C_H(S)$ for some $(k-1)$-clique $S$ of $H$;
\item $Q=\mathcal T_H(U)$ for some $(k+1)$-clique $U$ of $H$.
\end{enumerate}
In particular, every maximal clique of top type has size exactly $k+1$.
\end{corollary}

\begin{proof}
If $|Q|=1$, let $A$ be its unique vertex and choose any $(k-1)$-subset $S\subset A$.
If some $k$-clique other than $A$ also contained $S$, then it would be adjacent to $A$ in $\TJ{k}{H}$, contradicting maximality.
Hence, $Q=\mathcal C_H(S)$.
If $|Q|=2$, say $Q=\{A,B\}$, then $S=A\cap B$ has size $k-1$ and both $A$ and $B$ belong to $\mathcal C_H(S)$.
If $\mathcal C_H(S)$ contained a third member, then $Q$ would not be maximal.
Thus, again $Q=\mathcal C_H(S)$.
Assume now that $|Q|\ge 3$.
By \cref{thm:camera-TJ-clique-structure}, $Q$ is of top type or star type.
If $Q$ is of star type, then every vertex of $\mathcal C_H(S)$ is adjacent to all of its members, so maximality forces $Q=\mathcal C_H(S)$.
If $Q$ is of top type, then every set $U-u$ with $u\in U$ is adjacent to all its members, so maximality forces $Q=\mathcal T_H(U)$.
The size statement is immediate.
\end{proof}

\subsection{Basic Graph Families and Their Complements}
\label{sec:TJ-basic}

On the TJ side we begin with two elementary identifications.
\begin{proposition}\label{prop:TJ1}
For any graph $H$, we have $\TJ{1}{H}\cong K_{|V(H)|}$.
Consequently, a graph $G$ is a $\sfTJ_1$-graph if and only if $G$ is complete.
\end{proposition}

\begin{proof}
The vertices of $\TJ{1}{H}$ are the singletons of $V(H)$, and any two distinct singletons differ in exactly one element.
Hence, every two vertices are adjacent.
\end{proof}

\begin{lemma}\label{lem:TJ2-line-graph}
For every graph $H$, we have $\TJ{2}{H}\cong L(H)$.
\end{lemma}

\begin{proof}
The vertices of $\TJ{2}{H}$ are the edges of $H$.
Two edges are adjacent in $\TJ{2}{H}$ exactly when they share one endpoint, which is precisely the line-graph adjacency relation.
\end{proof}

The next lemma gives a triangle-free line-lift construction which will be useful for excluding certain boundary values on the TJ side.
\begin{lemma}\label{lem:TJ-line-lift}
Let $X$ be a triangle-free graph and let $k\ge 2$.
Define
\[
H_k(X)=
\begin{cases}
X, & k=2,\\
K_{k-2}\vee X, & k\ge 3,
\end{cases}
\]
where the added $(k-2)$-clique is disjoint from $V(X)$.
Then
\[
\TJ{k}{H_k(X)}\cong L(X).
\]
\end{lemma}

\begin{proof}
For $k=2$, this is exactly \cref{lem:TJ2-line-graph}.
Assume $k\ge 3$, and let $C$ be the added clique of size $k-2$.
Since $X$ is triangle-free, every clique of $X$ has size at most $2$.
Therefore, every $k$-clique of $H_k(X)=K_{k-2}\vee X$ contains all vertices of $C$ and exactly two adjacent vertices $u,v\in V(X)$.
Conversely, every edge $uv\in E(X)$ gives a $k$-clique $C+u+v$.
Hence, the map
\[
C+u+v\longmapsto uv
\]
is a bijection from the vertices of $\TJ{k}{H_k(X)}$ to the edges of $X$.
Two such $k$-cliques are TJ-adjacent exactly when the corresponding edges of $X$ share one endpoint.
Thus, the adjacency relation is precisely that of the line graph $L(X)$.
\end{proof}

For complete bipartite and friendship targets on the TJ side, we also need two local lemmas.

\begin{lemma}\label{lem:TJ-trianglefree-degree-bound}
Let $H$ be a graph, let $k\ge 2$, and let $G=\TJ{k}{H}$ be triangle-free.
Then every vertex of $G$ has degree at most $k$.
In particular,
\[
\Delta(G)\le k.
\]
\end{lemma}

\begin{proof}
Let $A$ be a vertex of $G$, viewed as a $k$-clique of $H$.
Any neighbor $B$ of $A$ satisfies $|A\cap B|=k-1$, so $B=A-a+x$ for a unique vertex $a\in A$ and some vertex $x\notin A$.
If two distinct neighbors $B_1,B_2$ of $A$ both deleted the same vertex $a$, then
\[
A\cap B_1=A\cap B_2=A-a,
\]
so $|B_1\cap B_2|=k-1$ and hence $B_1$ and $B_2$ would be adjacent in $G$.
Then $A,B_1,B_2$ would span a triangle, contradicting the triangle-freeness of $G$.
Thus, for each of the $k$ choices of $a\in A$, there is at most one neighbor of $A$ that deletes $a$.
Therefore, $\deg_G(A)\le k$.
\end{proof}

\begin{lemma}\label{lem:TJ-trianglefree-common-neighbors}
Let $H$ be a graph, let $k\ge 2$, and let $G=\TJ{k}{H}$ be triangle-free.
Then any two nonadjacent vertices of $G$ have at most two common neighbors.
\end{lemma}

\begin{proof}
Let $A,B$ be nonadjacent vertices of $G$.
If they have no common neighbor, there is nothing to prove.
So let $C$ be a common neighbor.
Since $|A\cap C|=|B\cap C|=k-1$, we have
\[
|A\cap B|\ge |A\cap C|+|B\cap C|-|C|=(k-1)+(k-1)-k=k-2.
\]
Because $A$ and $B$ are not adjacent, $|A\cap B|\neq k-1$, hence $|A\cap B|=k-2$.
Write
\[
A=S+a_1+a_2,\qquad B=S+b_1+b_2,
\]
where $S=A\cap B$ has size $k-2$ and $a_1,a_2,b_1,b_2$ are pairwise distinct.
Now let $D$ be any common neighbor of $A$ and $B$.
Since $|A\cap D|=|B\cap D|=k-1$, the set $D$ must contain $S$, exactly one of $\{a_1,a_2\}$, and exactly one of $\{b_1,b_2\}$.
Thus, every common neighbor of $A$ and $B$ is one of the four candidate sets
\[
S+a_i+b_j\qquad (i,j\in\{1,2\}).
\]
Among any three of these four candidates, two share the same choice of $a_i$ or the same choice of $b_j$.
Such two candidates intersect in exactly $k-1$ vertices, so they are adjacent in $G$.
But all common neighbors of $A$ and $B$ lie in $N_G(A)$, and $N_G(A)$ is an independent set because $G$ is triangle-free.
Therefore, at most two of the four candidates can occur.
\end{proof}

\begin{lemma}\label{lem:TJ-neighborhood-line-bipartite}
Let $H$ be a graph, let $k\ge 2$, and let $A$ be a vertex of $\TJ{k}{H}$.
Then there exists a bipartite graph $B_A$ such that
\[
N_{\TJ{k}{H}}(A)\cong L(B_A).
\]
\end{lemma}

\begin{proof}
Write $A=\{a_1,\dots,a_k\}$.
Let $X_A$ be the set of vertices $x\in V(H)\setminus A$ for which $A-a+x$ is a $k$-clique of $H$ for some $a\in A$.
Define a bipartite graph $B_A$ with bipartition $(A,X_A)$ by joining $a\in A$ to $x\in X_A$ whenever $A-a+x$ is a $k$-clique of $H$.
Then the vertices of $N_{\TJ{k}{H}}(A)$ are exactly the cliques $A-a+x$ corresponding to the edges $ax$ of $B_A$.
Moreover, two such neighbors $A-a+x$ and $A-b+y$ are adjacent in $\TJ{k}{H}$ if and only if they intersect in exactly $k-1$ vertices, which happens if and only if $a=b$ or $x=y$.
This criterion is precisely the adjacency relation in the line graph $L(B_A)$.
Hence, $N_{\TJ{k}{H}}(A)\cong L(B_A)$.
\end{proof}

\begin{lemma}\label{lem:line-bip-common-neighbors-clique}
Let $B$ be a bipartite graph.
Then for every adjacent pair of vertices $e,f$ in $L(B)$, the common neighborhood $N_{L(B)}(e)\cap N_{L(B)}(f)$ is a clique.
\end{lemma}

\begin{proof}
Let $e=xy_1$ and $f=xy_2$ be two adjacent edges of $B$ sharing the endpoint $x$.
Because $B$ is bipartite, every edge adjacent to both $e$ and $f$ must also be incident with $x$; otherwise it would join $y_1$ to $y_2$.
Thus, all common neighbors of $e$ and $f$ in $L(B)$ correspond to edges of $B$ incident with $x$.
Any two such edges are adjacent in $L(B)$, so the common neighborhood is a clique.
\end{proof}

\begin{theorem}\label{thm:warmup-tj}
The following exact formulas hold.
\begin{enumerate}[(1)]
\item For complete graphs,
\[
\KTJ(K_n)=\{k:k\ge 1\}\qquad (n\ge 1).
\]
\item For paths and cycles,
\[
\KTJ(P_2)=\{k:k\ge 1\},
\qquad
\KTJ(P_n)=\{k:k\ge 2\}\quad (n\ge 3),
\]
\[
\KTJ(C_3)=\{k:k\ge 1\},
\qquad
\KTJ(C_n)=\{k:k\ge 2\}\quad (n\ge 4).
\]
\item For stars and complete bipartite graphs,
\[
\KTJ(K_{1,n})=
\begin{cases}
\{k:k\ge 1\}, & n=1,\\
\{k:k\ge 2\}, & n=2,\\
\{k:k\ge n\}, & n\ge 3,
\end{cases}
\]
and
\[
\KTJ(K_{m,n})=\varnothing
\qquad (2\le m\le n\text{ and } n\ge 3).
\]
\item For book graphs,
\[
\KTJ(B_p)=
\begin{cases}
\{k:k\ge 1\}, & p=1,\\
\{2\}, & p=2,\\
\varnothing, & p\ge 3.
\end{cases}
\]
\item For friendship graphs,
\[
\KTJ(F_p)=
\begin{cases}
\{k:k\ge 1\}, & p=1,\\
\{k:k\ge p\}, & p\ge 2.
\end{cases}
\]
\end{enumerate}
\end{theorem}

\begin{proof}
For complete graphs, the case $k=1$ is \cref{prop:TJ1}.
Now fix $k\ge 2$.
Let $H$ be the complete $k$-partite graph with $k-1$ singleton parts and one part of size $n$.
Every $k$-clique of $H$ contains the $k-1$ singleton vertices together with one vertex from the large part, so there are exactly $n$ such cliques.
Any two intersect in exactly $k-1$ vertices, hence are adjacent in $\TJ{k}{H}$.
Therefore, $\TJ{k}{H}\cong K_n$.

For paths, $P_2=K_2$, so the complete-graph formula gives $\KTJ(P_2)=\{k:k\ge 1\}$.
Now let $n\ge 3$.
Since $P_n$ is not complete, \cref{prop:TJ1} gives $1\notin\KTJ(P_n)$.
Also $P_{n+1}$ is triangle-free and satisfies $L(P_{n+1})\cong P_n$.
Therefore, \cref{lem:TJ-line-lift} yields a witness for every $k\ge 2$, so
\[
\KTJ(P_n)=\{k:k\ge 2\}.
\]
For cycles, the case $C_3=K_3$ again follows from complete graphs.
If $n\ge 4$, then $C_n$ is not complete, so $1\notin \KTJ(C_n)$ by \cref{prop:TJ1}.
Moreover $C_n$ is triangle-free and $L(C_n)\cong C_n$, so \cref{lem:TJ-line-lift} gives
\[
\KTJ(C_n)=\{k:k\ge 2\}.
\]

For complete bipartite graphs, the first three cases are exactly
\[
K_{1,1}=P_2,\qquad K_{1,2}=P_3,\qquad K_{2,2}=C_4,
\]
so the corresponding formulas above already give the values for stars with $n\le 2$.
Now assume $2\le m\le n$ and $n\ge 3$.
Suppose to the contrary that $K_{m,n}\cong \TJ{k}{H}$ for some graph $H$ and some $k\ge 1$.
Because $K_{m,n}$ is not complete, \cref{prop:TJ1} gives $k\neq 1$, hence $k\ge 2$.
Choose two vertices in the part of size $m$.
They are nonadjacent and have exactly $n$ common neighbors.
Since $n\ge 3$, this contradicts \cref{lem:TJ-trianglefree-common-neighbors}.
Thus,
\[
\KTJ(K_{m,n})=\varnothing
\qquad (2\le m\le n\text{ and } n\ge 3).
\]
It remains to classify the stars $K_{1,n}$ with $n\ge 3$.
Because $K_{1,n}$ is not complete, \cref{prop:TJ1} gives $1\notin \KTJ(K_{1,n})$.
If $\TJ{k}{H}\cong K_{1,n}$, then \cref{lem:TJ-trianglefree-degree-bound} yields
\[
n=\Delta(K_{1,n})\le k,
\]
so no value $k<n$ is feasible.
Conversely, let $k\ge n$.
Start with a $k$-clique
\[
Q=\{q_1,\ldots,q_k\},
\]
and add vertices $x_1,\ldots,x_n$.
For each $i\in\{1,\ldots,n\}$, join $x_i$ to every vertex of $Q$ except $q_i$, and add no edges among the $x_i$.
Then the $k$-cliques of the resulting graph are exactly
\[
Q
\quad\text{and}\quad
Q-q_i+x_i\qquad (1\le i\le n).
\]
In $\TJ{k}{H}$, the central clique $Q$ is adjacent to each $Q-q_i+x_i$, while distinct leaves intersect in only $k-2$ vertices and are therefore nonadjacent.
Hence, $\TJ{k}{H}\cong K_{1,n}$.
This argument proves
\[
\KTJ(K_{1,n})=\{k:k\ge n\}
\qquad (n\ge 3).
\]

Next, consider book graphs.
The case $p=1$ is again $K_3$, so the complete-graph formula gives
\[
\KTJ(B_1)=\{k:k\ge 1\}.
\]
Now let $p\ge 2$.
Because $B_p$ is not complete, \cref{prop:TJ1} gives $1\notin \KTJ(B_p)$.
We first consider the case $k=2$.
If $p=2$, then $B_2$ is the diamond graph.
Let $X$ be the graph obtained from a triangle $abc$ by adding a new vertex $d$ adjacent only to $a$.
Then $L(X)\cong B_2$, and therefore \cref{lem:TJ2-line-graph} gives $2\in \KTJ(B_2)$.
Now assume that $p\ge 3$.
Suppose to the contrary that $B_p\cong \TJ{2}{H}$ for some graph $H$.
By \cref{lem:TJ2-line-graph}, this means $B_p\cong L(H)$.
Fix one of the two vertices of degree $p+1$ in $B_p$.
Its neighborhood contains the independent set formed by any three page vertices, so $B_p$ contains an induced claw.
On the other hand, if $e=xy$ is any edge of $H$, then the neighbors of $e$ in $L(H)$ are exactly the edges incident with $x$ or with $y$ other than $e$ itself.
Thus, the neighborhood of every vertex of $L(H)$ is the union of at most two cliques, and in particular $L(H)$ is claw-free.
This contradiction shows that $2\notin \KTJ(B_p)$ for all $p\ge 3$.

It remains to exclude every $k\ge 3$ when $p\ge 2$.
Suppose to the contrary that $B_p\cong \TJ{k}{H}$ for some graph $H$ and some $k\ge 3$.
Write the vertices of $B_p$ as $u,v,w_1,\dots,w_p$, where $uv$ is the common edge of the $p$ page-triangles.
Let $U$ and $V$ be the $k$-cliques of $H$ corresponding to $u$ and $v$.
For each $i\in\{1,\dots,p\}$, the triangle $\{u,v,w_i\}$ is a maximal clique of $B_p$ of size $3$.
Since every top clique of $\TJ{k}{H}$ has size exactly $k+1\ge 4$ by \cref{cor:camera-TJ-maximal-cliques}, each such triangle must be a star clique.
Hence, for each $i$, the clique corresponding to $w_i$ contains the common $(k-1)$-subset $U\cap V$.
Therefore, the vertices corresponding to $u,v,w_1,\dots,w_p$ all belong to the same star clique of $\TJ{k}{H}$.
In particular, the vertices corresponding to $w_1$ and $w_2$ are adjacent, contradicting the fact that $w_1$ and $w_2$ are nonadjacent in $B_p$.
Therefore,
\[
\KTJ(B_p)=
\begin{cases}
\{k:k\ge 1\}, & p=1,\\
\{2\}, & p=2,\\
\varnothing, & p\ge 3.
\end{cases}
\]

Finally, consider friendship graphs.
The case $p=1$ is again $K_3$.
Now let $p\ge 2$.
To show the upper bound, suppose $\TJ{k}{H}\cong F_p$ and let $C_0$ be the $k$-clique corresponding to the center of $F_p$.
This central vertex has degree $2p$.
Each neighbor of $C_0$ shares one of the $k$ different $(k-1)$-subsets of $C_0$.
By the pigeonhole principle, some $(k-1)$-subset is shared by at least $\lceil 2p/k\rceil$ neighbors.
Together with $C_0$, these cliques form a clique of size at least $\lceil 2p/k\rceil+1$ in $\TJ{k}{H}$.
Since $\omega(F_p)=3$, we must have $\lceil 2p/k\rceil+1\le 3$, hence $p\le k$.
So $\KTJ(F_p)\subseteq \{k:k\ge p\}$.
For the reverse inclusion, fix $k\ge p$.
Let $C=\{v_1,\ldots,v_k\}$ be a core clique, and for each $1\le i\le p$ let
\[
S_i=C\setminus\{v_i\}.
\]
Add two new vertices $w_{i,1},w_{i,2}$, each adjacent exactly to the vertices of $S_i$.
Then the $k$-cliques are exactly the core clique $C$ together with the $2p$ cliques
\[
Q_{i,j}=S_i\cup\{w_{i,j}\}
\qquad (1\le i\le p,\ j\in\{1,2\}).
\]
The clique $C$ is adjacent to all $Q_{i,j}$, each pair $Q_{i,1},Q_{i,2}$ is adjacent, and cliques from different indices intersect in only $k-2$ vertices and are therefore nonadjacent.
Thus, $\TJ{k}{H}\cong F_p$.
Hence, $\KTJ(F_p)=\{k:k\ge p\}$ for all $p\ge 2$.
\end{proof}

Before determining the complement families on the TJ side, we need the following proposition for disjoint unions of cliques.

\begin{proposition}\label{prop:camera-complement-clique-unions-TJ}
Let
\[
G=\bigsqcup_{i=1}^t K_{n_i}
\qquad (t\ge 1,\ n_i\ge 1).
\]
If $t\ge 2$, then
\[
\KTJ(G)=\{k:k\ge 2\}.
\]
\end{proposition}
\begin{proof}
If $t\ge 2$ then $G$ is not complete, so $1\notin \KTJ(G)$ by \cref{prop:TJ1}.
Now fix $k\ge 2$.
For each $i$, the complete-graph formula in \cref{thm:warmup-tj} gives a graph $H_i$ with $\TJ{k}{H_i}\cong K_{n_i}$.
Let $H:=\bigsqcup_{i=1}^t H_i$.
Every $k$-clique of $H$ lies in a single connected component of $H$, so
\[
\TJ{k}{H}\cong \bigsqcup_{i=1}^t \TJ{k}{H_i}\cong \bigsqcup_{i=1}^t K_{n_i}=G.
\]
Thus, every $k\ge 2$ belongs to $\KTJ(G)$.
\end{proof}

We are now ready to determine the corresponding complement basic-family results on the TJ side.
\begin{theorem}\label{thm:camera-tj-basic-complements}
The following exact formulas hold.
\begin{enumerate}[(1)]
\item
\[
\KTJ(\overline{K_n})=
\begin{cases}
\{k:k\ge 1\}, & n=1,\\
\{k:k\ge 2\}, & n\ge 2.
\end{cases}
\]
\item
\[
\KTJ(\overline{K_{1,n}})=\{k:k\ge 2\}
\qquad (n\ge 1).
\]
\item For $2\le m\le n$,
\[
\KTJ(\overline{K_{m,n}})=\{k:k\ge 2\}.
\]
\item
\[
\KTJ(\overline{B_p})=\{k:k\ge 2\}
\qquad (p\ge 1).
\]
\item
\[
\KTJ(\overline{P_n})=
\begin{cases}
\{k:k\ge 2\}, & 2\le n\le 5,\\
\varnothing, & n\ge 6.
\end{cases}
\]
\item
\[
\KTJ(\overline{C_n})=
\begin{cases}
\{k:k\ge 2\}, & 4\le n\le 6,\\
\varnothing, & n\ge 7.
\end{cases}
\]
\item
\[
\KTJ(\overline{F_p})=
\begin{cases}
\{k:k\ge 2\}, & p\le 2,\\
\{2\}, & p=3,\\
\varnothing, & p\ge 4.
\end{cases}
\]
\end{enumerate}
\end{theorem}

\begin{proof}
The case $\overline{K_1}=K_1$ in (1) follows from \cref{thm:warmup-tj}.
For the remaining cases in (1)--(4), the decompositions
\[
\overline{K_n}=nK_1,
\qquad
\overline{K_{1,n}}=K_n\sqcup K_1,
\qquad
\overline{K_{m,n}}=K_m\sqcup K_n,
\qquad
\overline{B_p}=K_p\sqcup 2K_1
\]
and \cref{prop:camera-complement-clique-unions-TJ} give the displayed formulas.

For paths, we have
\[
\overline{P_2}=2K_1,
\qquad
\overline{P_3}=K_2\sqcup K_1,
\qquad
\overline{P_4}=P_4,
\]
so the exact formulas for $2\le n\le 4$ follow from \cref{prop:camera-complement-clique-unions-TJ,thm:warmup-tj}.
For $\overline{P_5}$, let $R$ be the bipartite graph with bipartition $\{u,v\}\sqcup\{x,y,z\}$ and edge set
\[
ux,\ uy,\ uz,\ vx,\ vy.
\]
The graph $R$ is triangle-free, and its line graph is $\overline{P_5}$.
Hence, \cref{lem:TJ-line-lift} gives
\[
\{k:k\ge 2\}\subseteq \KTJ(\overline{P_5}).
\]
Since $\overline{P_5}$ is not complete, \cref{prop:TJ1} excludes $k=1$.
Therefore,
\[
\KTJ(\overline{P_5})=\{k:k\ge 2\}.
\]

We next show that $\overline{P_6}$ is not a line graph.
Label its vertices by $e_1,\dots,e_6$ so that $e_ie_{i+1}\notin E(\overline{P_6})$ for $1\le i\le 5$, and every other pair is adjacent.
If $\overline{P_6}\cong L(R')$, then $e_1$ and $e_2$ correspond to disjoint edges of $R'$, say $ab$ and $cd$ with $a,b,c,d$ pairwise distinct.
Because $e_4,e_5,e_6$ are adjacent to both $e_1$ and $e_2$, each must be one of the four cross-edges $ac,ad,bc,bd$.
Since $e_4$ and $e_5$ are nonadjacent in $\overline{P_6}$, they are disjoint as edges of $R'$, so after relabeling we may assume $e_4=ac$ and $e_5=bd$.
But then $e_6$ must be adjacent to $e_1$, $e_2$, and $e_4$, while remaining nonadjacent to $e_5$.
Among the four cross-edges, the only one disjoint from $bd$ is $ac$, which is already used by $e_4$.
This contradiction shows that $\overline{P_6}$ is not a line graph.
We use the elementary hereditary fact that induced subgraphs of line graphs are line graphs: if $Y=L(R)$ and $X\subseteq V(Y)$ corresponds to an edge set $F\subseteq E(R)$, then $Y[X]=L(R[F])$, where $R[F]$ is the subgraph of $R$ with edge set $F$.
In particular, $\overline{P_m}$ is not a line graph for every $m\ge 6$, because $\overline{P_m}$ contains an induced copy of $\overline{P_6}$ whenever $m\ge 6$.

Now let $n\ge 6$ and suppose $\overline{P_n}\cong \TJ{k}{H}$.
Since $\overline{P_n}$ is not complete, \cref{prop:TJ1} excludes $k=1$.
The case $k=2$ is impossible because $\overline{P_n}$ is not a line graph.
If $n\ge 8$, let $x$ be an endpoint of $P_n$.
Then
\[
N_{\overline{P_n}}(x)\cong \overline{P_{n-2}},
\]
and $n-2\ge 6$.
By \cref{lem:TJ-neighborhood-line-bipartite}, the neighborhood of the corresponding vertex of $\TJ{k}{H}$ must be a line graph.
This contradicts the previous paragraph.
Thus, only $n\in\{6,7\}$ remain.

Assume first that $n=6$.
The cliques $\{1,3,5\}$ and $\{1,3,6\}$ are maximal triangles of $\overline{P_6}$.
If $k\ge 3$, then every maximal top clique in $\TJ{k}{H}$ has size $k+1\ge 4$.
Therefore, each of these two triangles must be a star clique.
However, the two maximal triangles $\{1,3,5\}$ and $\{1,3,6\}$ of $\overline{P_6}$ share the edge $\{1,3\}$.
A star clique containing the edge $\{1,3\}$ is uniquely determined by the common $(k-1)$-subset of the corresponding adjacent vertices, so two distinct star cliques cannot share an edge.
This contradiction rules out every $k\ge 3$.

Assume next that $n=7$.
If $k\ge 4$, then the $4$-clique $\{1,3,5,7\}$ is maximal in $\overline{P_7}$, because each of $2,4,6$ misses at least one vertex of this clique.
This $4$-clique cannot be top-type, since every maximal top clique in $\TJ{k}{H}$ has size $k+1\ge 5$ by \cref{cor:camera-TJ-maximal-cliques}.
Therefore, it must be a star clique.
But the maximal triangle $\{1,3,6\}$ shares the edge $\{1,3\}$ with it, again impossible.
It remains to exclude $k=3$.
Let $A_1,A_3,A_5,A_7$ be the four vertices corresponding to the clique $\{1,3,5,7\}$.
This $4$-clique cannot be star-type. Indeed, otherwise there would exist a $2$-set $T$ and pairwise distinct elements $x_1,x_3,x_5,x_7\notin T$ such that
\[
A_1=T+x_1,\qquad A_3=T+x_3,\qquad A_5=T+x_5,\qquad A_7=T+x_7.
\]
Because vertex $2$ is adjacent exactly to $5$ and $7$ inside this $4$-clique, we would have
\[
A_2=\{u,x_5,x_7\}
\]
for some $u\in T$. Similarly, because vertex $6$ is adjacent exactly to $1$ and $3$ inside this $4$-clique, we would have
\[
A_6=\{v,x_1,x_3\}
\]
for some $v\in T$. Then $|A_2\cap A_6|\le 1$, contradicting the edge $26$ of $\overline{P_7}$. Thus, this $4$-clique must be top-type.
Hence, there exists a $4$-set $U=\{a,b,c,d\}$ such that
\[
A_1=U-a,\qquad A_3=U-b,\qquad A_5=U-c,\qquad A_7=U-d.
\]
Because vertex $2$ of $\overline{P_7}$ is adjacent exactly to $5$ and $7$ inside this $4$-clique, we have
\[
A_2=\{a,b,x_2\}
\]
for some $x_2\notin U$.
Similarly,
\[
A_4=\{b,c,x_4\},\qquad A_6=\{c,d,x_6\}
\]
for some $x_4,x_6\notin U$.
Since $2$ is adjacent to $4$, we have $|A_2\cap A_4|=2$, and therefore $x_2=x_4$.
Since $4$ is adjacent to $6$, we also have $x_4=x_6$.
Hence, $x_2=x_4=x_6$.
But then
\[
|A_2\cap A_6|=1,
\]
contrary to the edge $26$ of $\overline{P_7}$.
Therefore, $\KTJ(\overline{P_n})=\varnothing$ for all $n\ge 6$.

For cycles, we have
\[
\overline{C_4}=2K_2
\qquad\text{and}\qquad
\overline{C_5}=C_5,
\]
Thus the first two exact formulas follow from \cref{prop:camera-complement-clique-unions-TJ} and \cref{thm:warmup-tj}.
Also $\overline{C_6}$ is the triangular prism, that is, the line graph of the triangle-free graph $K_{3,2}$.
Hence, \cref{lem:TJ-line-lift} gives
\[
\{k:k\ge 2\}\subseteq \KTJ(\overline{C_6}).
\]
Since $\overline{C_6}$ is not complete, \cref{prop:TJ1} excludes $k=1$.
Therefore,
\[
\KTJ(\overline{C_6})=\{k:k\ge 2\}.
\]
Now let $n\ge 7$ and suppose $\overline{C_n}\cong \TJ{k}{H}$.
Since $\overline{C_n}$ is not complete, \cref{prop:TJ1} excludes $k=1$.
The case $k=2$ is impossible because $\overline{C_n}$ contains an induced copy of $\overline{P_6}$.
If $n\ge 9$, fix any vertex $x$ of $\overline{C_n}$.
Then
\[
N_{\overline{C_n}}(x)\cong \overline{P_{n-3}},
\]
and $n-3\ge 6$.
By \cref{lem:TJ-neighborhood-line-bipartite}, this neighborhood must be a line graph, contradiction.
Thus, only $n\in\{7,8\}$ remain.

Assume first that $n=7$.
The cliques $\{1,3,5\}$ and $\{1,3,6\}$ are maximal triangles of $\overline{C_7}$.
If $k\ge 3$, each of these maximal triangles in $\TJ{k}{H}$ must be a star clique.
However, the two maximal triangles $\{1,3,5\}$ and $\{1,3,6\}$ of $\overline{C_7}$ share the edge $\{1,3\}$, which is impossible for two distinct star cliques.
This contradiction excludes every $k\ge 3$.

Assume next that $n=8$.
If $k\ge 4$, then the $4$-clique $\{1,3,5,7\}$ is maximal in $\overline{C_8}$, because each of $2,4,6,8$ misses at least one vertex of this clique.
This $4$-clique cannot be top-type, since every maximal top clique in $\TJ{k}{H}$ has size $k+1\ge 5$ by \cref{cor:camera-TJ-maximal-cliques}.
Therefore, it must be a star clique, while the maximal triangle $\{1,3,6\}$ shares the edge $\{1,3\}$ with it.
This configuration is impossible.
It remains to exclude $k=3$.
Let $A_1,A_3,A_5,A_7$ be the four vertices corresponding to the clique $\{1,3,5,7\}$.
This $4$-clique cannot be star-type. Indeed, otherwise there would exist a $2$-set $T$ and pairwise distinct elements $x_1,x_3,x_5,x_7\notin T$ such that
\[
A_1=T+x_1,\qquad A_3=T+x_3,\qquad A_5=T+x_5,\qquad A_7=T+x_7.
\]
Because vertex $2$ is adjacent exactly to $5$ and $7$ inside this $4$-clique, while vertex $6$ is adjacent exactly to $1$ and $3$, we would have
\[
A_2=\{u,x_5,x_7\},\qquad A_6=\{v,x_1,x_3\}
\]
for some $u,v\in T$. Then $|A_2\cap A_6|\le 1$, contradicting the edge $26$ of $\overline{C_8}$. Thus, this $4$-clique must be top-type, so there exists a $4$-set $U=\{a,b,c,d\}$ such that
\[
A_1=U-a,\qquad A_3=U-b,\qquad A_5=U-c,\qquad A_7=U-d.
\]
Because vertex $2$ is adjacent exactly to $5$ and $7$ inside this clique, while vertex $6$ is adjacent exactly to $1$ and $3$, we obtain
\[
A_2=\{a,b,x_2\},\qquad A_6=\{c,d,x_6\}
\]
for some $x_2,x_6\notin U$.
Then $|A_2\cap A_6|\le 1$, contradicting the edge $26$ of $\overline{C_8}$.
Therefore, $\KTJ(\overline{C_n})=\varnothing$ for all $n\ge 7$.

For friendship complements, note that
\[
\overline{F_p}=K_1\sqcup \CP{p}.
\]
For $p=1$, we have $\overline{F_1}=3K_1$, so (1) gives $\KTJ(\overline{F_1})=\{k:k\ge 2\}$.
For $p=2$, fix $k\ge 2$.
Let $H_1$ be any graph with $\TJ{k}{H_1}\cong K_1$, given by \cref{thm:warmup-tj}.
Because $C_4$ is triangle-free and satisfies $L(C_4)\cong C_4$, \cref{lem:TJ-line-lift} gives a graph $H_2$ with $\TJ{k}{H_2}\cong C_4$.
Then
\[
\TJ{k}{H_1\sqcup H_2}\cong K_1\sqcup C_4=\overline{F_2},
\]
so $\{k:k\ge 2\}\subseteq \KTJ(\overline{F_2})$.
Since $\overline{F_2}$ is not complete, \cref{prop:TJ1} excludes $k=1$.
Therefore,
\[
\KTJ(\overline{F_2})=\{k:k\ge 2\}.
\]
For $p=3$, the value $2$ is feasible. Indeed, if $H=K_2\sqcup K_4$, then
\[
\TJ{2}{H}=L(H)=K_1\sqcup L(K_4)=K_1\sqcup \CP{3}=\overline{F_3}.
\]
Since $\overline{F_3}$ is not complete, \cref{prop:TJ1} excludes $k=1$.
Now suppose $\overline{F_3}\cong \TJ{k}{H}$ with $k\ge 3$.
Its nontrivial component is $\CP{3}$, and $\omega(\CP{3})=3$.
Hence, every maximal triangle of $\CP{3}$ would have to be a star clique.
But every edge of $\CP{3}$ lies in two triangles, whereas two distinct star cliques cannot share an edge.
This contradiction shows that $k\ge 3$ is impossible, and therefore $\KTJ(\overline{F_3})=\{2\}$.

Finally, let $p\ge 4$ and suppose $\overline{F_p}\cong \TJ{k}{H}$.
Because $\overline{F_p}$ is not complete, \cref{prop:TJ1} gives $k\ne 1$.
Hence, $k\ge 2$.
Fix a nonisolated vertex $x$ of $\overline{F_p}$.
Then $x$ lies in the $\CP{p}$-component, so
\[
N_{\overline{F_p}}(x)\cong \CP{p-1}.
\]
By \cref{lem:TJ-neighborhood-line-bipartite}, this neighborhood must be a line graph of a bipartite graph.
On the other hand, choose any adjacent pair in $\CP{p-1}$.
Its common neighborhood contains one full matched pair of $\CP{p-1}$, and hence is not a clique because $p-1\ge 3$.
This contradicts \cref{lem:line-bip-common-neighbors-clique}.
Therefore, $\KTJ(\overline{F_p})=\varnothing$ for all $p\ge 4$.
This completes the proof.
\end{proof}

\subsection{Johnson Graphs}
\label{sec:TJ-Johnson}

We now turn to Johnson graphs on the TJ side.
The next results determine $\KTJ(J(n,r))$ for every nontrivial Johnson graph $J(n,r)$.
The section is organized by the cases appearing in the final theorem: first the complete-graph levels $r=1$ and $r=n-1$, then the stable rank-$2$ family, the stable higher-rank case $n>2r\ge 6$, and finally the boundary case $n=2r\ge 4$.

\begin{theorem}\label{thm:camera-tj-johnson}
For every pair of integers $n\ge 2$ and $1\le r\le n-1$, let
\[
s:=\min\{r,n-r\}.
\]
Then
\[
\KTJ(J(n,r))=
\begin{cases}
\{k:k\ge 1\}, & s=1,\\
\{s\}, & s\ge 2\text{ and }n=2s,\\
\{s,n-s\}, & s\ge 2\text{ and }n>2s.
\end{cases}
\]
\end{theorem}

The proof is deferred until the end of the subsection, after the complete-graph case, the stable rank-$2$ case, the stable higher-rank case, and the boundary case have been established.

The remaining Johnson/TJ arguments follow a common pattern.
We first identify the relevant maximum cliques of the Johnson graph and how many of them pass through a fixed vertex.
We then compare their size with the top-clique size $k+1$ from \cref{cor:camera-TJ-maximal-cliques}.
Whenever the Johnson maximum cliques are too large to be top cliques in a TJ witness, they must all be realized as star cliques.
At that point we either contradict \cref{lem:camera-TJ-clique-at-most-k-stars} directly, or apply the common-core lemmas to reconstruct a rigid family inside the witness.

The next lemma bounds the number of star cliques through a vertex.
\begin{lemma}\label{lem:camera-TJ-clique-at-most-k-stars}
Let $H$ be a graph, let $k\ge 2$, and let $G=\TJ{k}{H}$.
Then every vertex of $G$ lies in at most $k$ distinct star cliques of $G$.
\end{lemma}

\begin{proof}
Let $Q$ be a $k$-clique of $H$, viewed as a vertex of $G$.
If a star clique of $G$ contains $Q$, then by definition it has the form $\mathcal C_H(S)$ for some $(k-1)$-clique $S\subseteq Q$.
Distinct star cliques containing $Q$ correspond to distinct $(k-1)$-subcliques of $Q$.
Since a $k$-clique has exactly $k$ such subcliques, the vertex $Q$ lies in at most $k$ distinct star cliques.
\end{proof}

The star-count comparison used in the small-level exclusion is shown in \cref{fig:tj-star-count-obstruction}.
\begin{figure}[ht]
\centering
\begin{adjustbox}{max width=\textwidth}
\begin{tikzpicture}[x=1.05cm,y=1.05cm]
\tikzset{
  title/.style={font=\bfseries\small, align=center, minimum width=4.80cm,
    minimum height=.44cm, text width=4.60cm},
  note/.style={font=\footnotesize, align=center, minimum width=5.00cm,
    minimum height=.40cm, text width=4.70cm},
  smallnote/.style={font=\scriptsize, align=center, minimum width=2.30cm,
    minimum height=.34cm, text width=2.10cm},
  vertexpanel/.style={rectangle, draw=black, fill=white, line width=.7pt,
    minimum width=2.80cm, minimum height=.58cm, text width=2.60cm, inner sep=2.5pt,
    font=\footnotesize, align=center},
  familypanel/.style={rectangle, draw=black, fill=white, line width=.55pt,
    minimum width=1.55cm, minimum height=.52cm, text width=1.32cm, inner sep=2pt,
    font=\scriptsize, align=center},
  ellipsis/.style={font=\footnotesize, minimum width=.70cm, minimum height=.30cm,
    text width=.65cm, align=center},
  arrow/.style={->, draw=black, line width=.55pt},
  separator/.style={densely dashed, draw=black, line width=.45pt},
  countpanel/.style={rectangle, draw=black, fill=white, line width=.65pt,
    minimum width=4.00cm, minimum height=.62cm, text width=3.70cm, inner sep=3pt,
    font=\footnotesize, align=center}
}
\path[use as bounding box] (-.20,-.55) rectangle (12.20,5.65);

\node[title] at (2.55,5.28) {(a) one vertex in\\$\mathsf{TJ}_k(H)$};
\node[vertexpanel] (Q) at (2.55,4.28) {$Q=\{q_1,\ldots,q_k\}$};

\node[familypanel] (S1) at (.62,2.62) {$Q-q_1$};
\node[familypanel] (S2) at (1.90,2.62) {$Q-q_2$};
\node[ellipsis] at (3.05,2.62) {$\cdots$};
\node[familypanel] (Sk) at (4.45,2.62) {$Q-q_k$};

\foreach \x in {S1,S2,Sk} {
  \draw[arrow] (Q.south) -- (\x.north);
}
\node[note] at (2.55,1.72) {cores of star cliques through $Q$};
\node[countpanel] at (2.55,.72) {at most $k$ star cliques};

\draw[separator] (5.88,.42) -- (5.88,5.25);

\node[title] at (9.10,5.28) {(b) one vertex in $J(n,r)$};
\node[vertexpanel] (A) at (9.10,4.28) {$A=\{a_1,\ldots,a_r\}$};

\node[familypanel] (M1) at (7.05,2.62) {$\mathcal M_{A-a_1}$};
\node[familypanel] (M2) at (8.42,2.62) {$\mathcal M_{A-a_2}$};
\node[ellipsis] at (9.68,2.62) {$\cdots$};
\node[familypanel] (Mr) at (11.12,2.62) {$\mathcal M_{A-a_r}$};

\foreach \x in {M1,M2,Mr} {
  \draw[arrow] (A.south) -- (\x.north);
}
\node[note] at (9.10,1.72) {maximum Johnson stars through $A$};
\node[countpanel] at (9.10,.72) {exactly $r$ maximum cliques};

\node[note] at (5.88,-.20) {$k<r$: available $<{}$ required};

\end{tikzpicture}
\end{adjustbox}
\caption{The star-count obstruction used for small TJ levels of stable Johnson graphs.  A vertex $Q$ of $\mathsf{TJ}_k(H)$ lies in at most one star clique for each $(k-1)$-subclique $Q-q_i$.  In contrast, a vertex $A$ of $J(n,r)$ lies in exactly the $r$ maximum Johnson stars indexed by $A-a_i$.  The notation $\mathcal M_X$ is introduced formally in the proof of \cref{thm:camera-tj-no-middle}; here $\mathcal M_{A-a_i}$ denotes the Johnson star with core $A-a_i$.  The arrows indicate which deleted subclique or Johnson star is indexed by the displayed vertex, not an edge of a reconfiguration graph.  When these Johnson maximum cliques are too large to be top cliques in a hypothetical TJ witness, all of them must become star cliques through the image of $A$, contradicting $k<r$.}
\label{fig:tj-star-count-obstruction}
\end{figure}
\FloatBarrier

The complete-graph Johnson cases are again exactly the complete-graph witnesses.
\begin{corollary}\label{cor:camera-johnson-complete-levels-tj}
For every integer $n\ge 2$,
\[
\KTJ(J(n,1))=\KTJ(J(n,n-1))=\{k:k\ge 1\}.
\]
\end{corollary}

\begin{proof}
Since
\[
J(n,1)\cong K_n\cong J(n,n-1),
\]
the claim is exactly the complete-graph formula from \cref{thm:warmup-tj}.
\end{proof}

We next settle the rank-$2$ Johnson family on the TJ side.
The maximum-clique statement is the rank-$2$ case of the standard star/top classification of Johnson cliques~\cite[Thm.~3.4 and Cor.~3.5]{ShuldinerO22}.
\begin{lemma}\label{lem:camera-Jn2-max-cliques-TJ}
Let $n\ge 5$. In the triangular graph $J(n,2)$, the maximum cliques are exactly the $n$ Johnson star cliques
\[
\mathcal M_i:=\bigl\{\{i,j\}: j\in [n]\setminus\{i\}\bigr\}
\qquad (i\in [n]).
\]
Each $\mathcal M_i$ has size $n-1$. Every vertex $\{i,j\}$ of $J(n,2)$ belongs to exactly the two maximum cliques $\mathcal M_i$ and $\mathcal M_j$, and for $i\ne j$ we have
\[
\mathcal M_i\cap \mathcal M_j=\bigl\{\{i,j\}\bigr\}.
\]
\end{lemma}

\begin{lemma}\label{lem:camera-TJ-star-family-common-core}
Let $H$ be a graph, let $k\ge 2$, and let $S_1,\dots,S_m$ be distinct $(k-1)$-cliques of $H$ such that for all distinct $i,j$ the union $S_i\cup S_j$ is a $k$-clique of $H$, and the $k$-clique $S_i\cup S_j$ contains no $S_\ell$ with $\ell\notin\{i,j\}$. Then there exist a $(k-2)$-clique $T$ of $H$ and pairwise distinct vertices $x_1,\dots,x_m\in V(H)\setminus T$ such that
\[
S_i=T+x_i\qquad (1\le i\le m).
\]
\end{lemma}

\begin{proof}
If $m=1$, then choose any vertex $x_1\in S_1$ and set
\[
T:=S_1-x_1.
\]
Then $T$ is a $(k-2)$-clique of $H$ and $S_1=T+x_1$.
So assume from now on that $m\ge 2$.
Fix two distinct indices, say $1$ and $2$, and write
\[
S_1=T+a,\qquad S_2=T+b,
\]
where $T=S_1\cap S_2$ has size $k-2$ and $a\ne b$.
Let $i\ge 3$.
Because $S_i\cup S_1$ and $S_i\cup S_2$ are both $k$-cliques, we have
\[
|S_i\cap S_1|=|S_i\cap S_2|=k-2.
\]
Since $S_1=T+a$ and $S_2=T+b$, the set $S_i$ can differ from each of them by exactly one vertex.
Therefore, either
\[
S_i=T+c
\]
for some vertex $c\notin T\cup\{a,b\}$, or else
\[
S_i=(T-x)+a+b
\]
for some $x\in T$.
In the second case we would have $S_i\subseteq S_1\cup S_2$, contradicting the assumption that the $k$-clique $S_1\cup S_2$ contains no $S_\ell$ with $\ell\notin\{1,2\}$.
Hence, every $S_i$ has the form $T+c$.
Writing $x_i$ for its unique vertex outside $T$, we obtain $S_i=T+x_i$ for all $i$.
The vertices $x_i$ are pairwise distinct because the sets $S_i$ are distinct.
Since $T\subseteq S_1$, it is a clique of $H$.
\end{proof}

\begin{theorem}\label{thm:camera-KTJ-Jn2}
For every $n\ge 5$,
\[
\KTJ(J(n,2))=\{2,n-2\}.
\]
\end{theorem}

\begin{proof}
Because
\[
\TJ{2}{K_n}\cong J(n,2)
\quad\text{and}\quad
\TJ{n-2}{K_n}\cong J(n,n-2)\cong J(n,2),
\]
the values $2$ and $n-2$ belong to $\KTJ(J(n,2))$.
For the converse, let $k\in \KTJ(J(n,2))$ and choose a graph $H$ with
\[
\TJ{k}{H}\cong J(n,2).
\]
Fix an isomorphism $\phi:J(n,2)\to\TJ{k}{H}$.
Because $J(n,2)$ is not complete when $n\ge 5$, \cref{prop:TJ1} gives $k\ne 1$.
If $k\in\{2,n-2\}$, we are done, so assume from now on that
\[
k\ge 3
\qquad\text{and}\qquad
k\ne n-2.
\]
By \cref{lem:camera-Jn2-max-cliques-TJ}, the graph $J(n,2)$ has exactly $n$ maximum cliques, each of size $n-1$, and every vertex lies in exactly two of them.
Because isomorphisms preserve maximal cliques, the $\phi$-image of each maximum clique of $J(n,2)$ is a maximal clique of $\TJ{k}{H}$. By \cref{cor:camera-TJ-maximal-cliques}, each such image is therefore either a star clique or a top clique. A top clique in $\TJ{k}{H}$ has size exactly $k+1$, and because $k+1\ne n-1$, no maximum clique of $J(n,2)$ can be realized as a top clique in the witness. Hence, every Johnson maximum clique is realized as a star clique of $\TJ{k}{H}$.
Let $S_1,\dots,S_n$ be the $(k-1)$-cliques of $H$ corresponding to the $n$ maximum cliques of $J(n,2)$ from \cref{lem:camera-Jn2-max-cliques-TJ}.
The cliques $S_1,\dots,S_n$ are pairwise distinct, because their star cliques are the distinct $\phi$-images of the Johnson maximum cliques.
For distinct $i,j$, the cliques $\mathcal C_H(S_i)$ and $\mathcal C_H(S_j)$ intersect in exactly one vertex of $\TJ{k}{H}$, because the corresponding maximum cliques of $J(n,2)$ do.
Any vertex in this intersection is a $k$-clique containing both $S_i$ and $S_j$.
Since $S_i$ and $S_j$ are distinct $(k-1)$-subcliques of that same $k$-clique, their union has size $k$, so the vertex necessarily equals $S_i\cup S_j$.
Thus, $S_i\cup S_j$ is a $k$-clique of $H$ for every $i\ne j$.
Moreover, each vertex of $J(n,2)$ belongs to exactly two maximum cliques by \cref{lem:camera-Jn2-max-cliques-TJ}, so for every $i\ne j$ the $k$-clique $S_i\cup S_j$ contains no $S_\ell$ with $\ell\notin\{i,j\}$.
Hence, \cref{lem:camera-TJ-star-family-common-core} applies.
We obtain a common $(k-2)$-clique $T$ and pairwise distinct vertices $x_1,\dots,x_n$ such that
\[
S_i=T+x_i\qquad (1\le i\le n).
\]
Then, for every $i\ne j$,
\[
S_i\cup S_j=T+x_i+x_j
\]
is a $k$-clique of $H$.
In particular, every pair of vertices among $T\cup\{x_1,\dots,x_n\}$ is adjacent in $H$, so this set induces a clique of size
\[
|T|+n=(k-2)+n=n+k-2.
\]
Consequently,
\[
|V(J(n,2))|=|V(\TJ{k}{H})|\ge \binom{n+k-2}{k}.
\]
But since $k\ge 3$ and $n\ge 5$,
\[
\binom{n+k-2}{k}=\binom{n+k-2}{n-2}\ge \binom{n+1}{n-2}=\binom{n+1}{3}>\binom{n}{2}=|V(J(n,2))|,
\]
a contradiction.
Therefore, $k\in\{2,n-2\}$.
Therefore, $\KTJ(J(n,2))=\{2,n-2\}$.
\end{proof}

The rank-$2$ incidence-transfer step used in the proof of \cref{thm:camera-KTJ-Jn2} is shown in \cref{fig:jn2-tj-star-core}.
\begin{figure}[ht]
\centering
\begin{adjustbox}{max width=\textwidth}
\begin{tikzpicture}[x=1cm,y=1cm]
\providecommand{\cref}[1]{\ref{#1}}
\tikzset{
  paneltitle/.style={font=\bfseries\small, align=center,
    minimum width=4.60cm, minimum height=.42cm, text width=4.45cm},
  cliquenode/.style={rectangle, draw=black, fill=white, line width=.6pt,
    font=\footnotesize, align=center, inner sep=2pt},
  edgelabel/.style={fill=white, inner sep=1.5pt, font=\scriptsize, align=center},
  note/.style={font=\scriptsize, align=center, minimum width=2.75cm,
    minimum height=.38cm, text width=2.60cm},
  formula/.style={rectangle, draw=black, fill=white, line width=.6pt,
    font=\scriptsize, align=center, inner sep=3pt},
  core/.style={rectangle, draw=black, fill=white, line width=.7pt,
    font=\small, align=center, inner sep=2pt},
  xnode/.style={circle, draw=black, fill=white, line width=.55pt,
    font=\scriptsize, align=center, inner sep=1.5pt},
  arrow/.style={->, draw=black, line width=.6pt},
  edge/.style={draw=black, line width=.55pt},
  faintedge/.style={draw=black, densely dotted, line width=.45pt},
  separator/.style={densely dashed, draw=black, line width=.45pt}
}
\path[use as bounding box] (-.15,-.80) rectangle (15.90,9.08);

\draw[separator] (6.55,4.58) -- (6.55,8.34);
\draw[separator] (.25,4.00) -- (15.55,4.00);

\node[paneltitle] at (3.28,8.46) {Johnson incidence\\in $J(n,2)$};
\node[cliquenode, minimum width=1.48cm, minimum height=.48cm, text width=1.18cm]
  (Mi) at (1.65,7.35) {$\mathcal M_i$};
\node[cliquenode, minimum width=1.48cm, minimum height=.48cm, text width=1.18cm]
  (Mj) at (4.82,7.35) {$\mathcal M_j$};
\node[cliquenode, minimum width=1.48cm, minimum height=.48cm, text width=1.18cm]
  (Ml) at (3.24,5.82) {$\mathcal M_\ell$};
\node[cliquenode, minimum width=1.48cm, minimum height=.48cm, text width=1.18cm]
  (Mo) at (5.72,5.82) {$\cdots$};

\draw[edge] (Mi) -- (Mj);
\draw[edge] (Mi) -- (Ml);
\draw[edge] (Mj) -- (Ml);
\draw[faintedge] (Mi) -- (Mo);
\draw[faintedge] (Mj) -- (Mo);
\draw[faintedge] (Ml) -- (Mo);

\node[edgelabel] at (3.24,7.63) {$\{i,j\}$};
\node[edgelabel] at (1.75,6.48) {$\{i,\ell\}$};
\node[edgelabel] at (4.73,6.48) {$\{j,\ell\}$};
\node[note] at (3.25,4.88) {edge labels are unique intersections};

\node[paneltitle, minimum width=6.20cm, text width=6.00cm]
  at (11.20,8.46) {same incidence in $\mathsf{TJ}_k(H)$};
\node[cliquenode, minimum width=2.18cm, minimum height=.50cm, text width=1.92cm]
  (Ci) at (8.28,7.35) {$\mathcal C_H(S_i)$};
\node[cliquenode, minimum width=2.18cm, minimum height=.50cm, text width=1.92cm]
  (Cj) at (12.84,7.35) {$\mathcal C_H(S_j)$};
\node[cliquenode, minimum width=2.18cm, minimum height=.50cm, text width=1.92cm]
  (Cl) at (10.56,5.82) {$\mathcal C_H(S_\ell)$};
\node[cliquenode, minimum width=2.18cm, minimum height=.50cm, text width=1.92cm]
  (Co) at (14.70,5.82) {$\cdots$};

\draw[edge] (Ci) -- (Cj);
\draw[edge] (Ci) -- (Cl);
\draw[edge] (Cj) -- (Cl);
\draw[faintedge] (Ci) -- (Co);
\draw[faintedge] (Cj) -- (Co);
\draw[faintedge] (Cl) -- (Co);

\node[edgelabel] at (10.56,7.72) {$Q_{ij}$};
\node[edgelabel] at (8.86,6.45) {$Q_{i\ell}$};
\node[edgelabel] at (12.25,6.45) {$Q_{j\ell}$};

\node[formula, minimum width=5.80cm, minimum height=.52cm, text width=5.45cm]
  (Qdef) at (10.55,5.12)
  {$Q_{ab}:=\phi(\{a,b\})=S_a\cup S_b$};
\node[formula, minimum width=4.05cm, minimum height=.50cm, text width=3.72cm]
  (nothird) at (13.26,4.47)
  {$S_c\not\subseteq Q_{ab}$ if $c\notin\{a,b\}$};

\draw[arrow] (5.95,7.35) -- (7.05,7.35)
  node[midway, above, edgelabel] {$\phi$};
\node[note, minimum width=3.05cm, text width=2.90cm]
  at (6.50,6.72) {max cliques map to TJ stars};

\node[paneltitle, minimum width=3.20cm, text width=3.05cm]
  at (2.00,3.26) {inside $H$};

\node[formula, minimum width=3.40cm, minimum height=.74cm, text width=3.10cm]
  (lemma) at (2.78,2.50)
  {\cref{lem:camera-TJ-star-family-common-core}:\\$S_a=T+x_a$ for all $a$};
\node[formula, minimum width=2.92cm, minimum height=.58cm, text width=2.62cm]
  (qab) at (6.34,2.50)
  {$Q_{ab}=T+x_a+x_b$};
\draw[arrow] (lemma.east) -- (qab.west);

\draw[edge] (9.15,.10) rectangle (15.38,3.07);
\node[edgelabel] at (12.27,3.28) {clique in $H$};

\node[core, minimum width=1.52cm, minimum height=.64cm, text width=1.20cm]
  (Tcore) at (10.18,1.54) {$T$};
\node[note, minimum width=1.80cm, text width=1.65cm]
  at (10.18,.82) {common core};

\node[xnode, minimum width=.70cm, minimum height=.70cm, text width=.52cm]
  (xi) at (12.02,2.40) {$x_i$};
\node[xnode, minimum width=.70cm, minimum height=.70cm, text width=.52cm]
  (xj) at (13.28,2.07) {$x_j$};
\node[xnode, minimum width=.70cm, minimum height=.70cm, text width=.52cm]
  (xl) at (13.28,1.00) {$x_\ell$};
\node[xnode, minimum width=.70cm, minimum height=.70cm, text width=.52cm]
  (xn) at (12.02,.67) {$x_n$};
\node[edgelabel] at (14.28,1.54) {$\cdots$};

\draw[edge] (xi.east) -- (xj.west);
\draw[edge] (xj.south) -- (xl.north);
\draw[edge] (xl.west) -- (xn.east);
\draw[edge] (xn.north) -- (xi.south);
\draw[edge] (xi.south east) -- (xl.north west);
\draw[edge] (xj.south west) -- (xn.north east);
\draw[edge] (Tcore.east) -- (xi.west);
\draw[edge] (Tcore.east) -- (xj.west);
\draw[edge] (Tcore.east) -- (xl.west);
\draw[edge] (Tcore.east) -- (xn.west);

\draw[arrow] (qab.east) -- (9.10,2.34);
\node[note, minimum width=5.08cm, minimum height=.42cm, text width=4.78cm]
  at (12.32,-.30)
  {$T\cup\{x_1,\ldots,x_n\}$ has size $n+k-2$};
\end{tikzpicture}
\end{adjustbox}
\caption{The rank-$2$ Johnson/TJ obstruction.  The top two panels are incidence schematics of clique families: Johnson maximum cliques $\mathcal M_a$ correspond under the assumed isomorphism $\phi$ to TJ star cliques $\mathcal C_H(S_a)$.  Each label $Q_{ab}:=\phi(\{a,b\})=S_a\cup S_b$ is one vertex of $\mathsf{TJ}_k(H)$, equivalently one $k$-clique of the underlying graph $H$; the no-third condition records that $S_c\nsubseteq Q_{ab}$ when $c\notin\{a,b\}$.  \cref{lem:camera-TJ-star-family-common-core} then gives $S_a=T+x_a$ for all $a$, so $T\cup\{x_1,\ldots,x_n\}$ is a clique of size $n+k-2$ in $H$, producing, in the proof setting $n\ge 5$ and $k\ge 3$, at least $\displaystyle \binom{n+k-2}{k}>\binom n2$ TJ vertices.  The arrows denote proof dependence, not graph edges: the $\phi$ arrow is the assumed isomorphism, the next arrow substitutes the lemma output into $Q_{ab}=S_a\cup S_b$, and the final arrow uses all pairs $a,b$ to form the clique in $H$.}
\label{fig:jn2-tj-star-core}
\end{figure}
\FloatBarrier

We now turn to the stable higher-rank regime.
The next corollary is the simplest instance of the template above: stable Johnson maximum cliques are larger than TJ top cliques, so the star-count bound gives an immediate contradiction.
\begin{corollary}\label{cor:camera-JnR-no-small-TJ}
Assume $n>2r\ge 4$.
Then
\[
\{1,\dots,r-1\}\cap \KTJ(J(n,r))=\varnothing.
\]
\end{corollary}

\begin{proof}
Suppose to the contrary that $J(n,r)\cong \TJ{k}{H}$ for some $1\le k\le r-1$.
Fix an isomorphism $\phi:J(n,r)\to\TJ{k}{H}$.
Because $J(n,r)$ is not complete when $n>2r\ge 4$, \cref{prop:TJ1} gives $k\ne 1$.
Hence, $2\le k\le r-1$.
If $r=2$, this is already impossible; hence assume $r\ge 3$.
By \cref{lem:camera-JnR-max-stars-stable}, every maximum clique of $J(n,r)$ is a Johnson star of size $n-r+1$, and every vertex lies in exactly $r$ such maximum cliques. Because isomorphisms preserve maximal cliques, the $\phi$-image of each maximum clique of $J(n,r)$ is a maximal clique of $\TJ{k}{H}$. By \cref{cor:camera-TJ-maximal-cliques}, each such image is therefore either a star clique or a top clique. A top clique in $\TJ{k}{H}$ has size exactly $k+1$, and since $k+1\le r<n-r+1$, no maximum clique of $J(n,r)$ can be a top clique in the witness. Hence, every maximum clique of $J(n,r)$ must be realized as a star clique in $\TJ{k}{H}$.
Fix a vertex $A$ of $J(n,r)$.
The vertex $\phi(A)$ of $\TJ{k}{H}$ therefore lies in at least $r$ distinct star cliques.
But this contradicts \cref{lem:camera-TJ-clique-at-most-k-stars}, because $r>k$.
Therefore, no such $k$ exists.
\end{proof}

The next lemma gives the base case for the common-core normalization.
\begin{lemma}\label{lem:camera-TJ-pair-star-common-core}
Let $H$ be a graph, let $k\ge 3$, let $n\ge 4$, and let
\[
\{S_{ij}:1\le i<j\le n\}
\]
be distinct $(k-1)$-cliques of $H$ such that for every triple $1\le i<j<\ell\le n$ there exists a $k$-clique $Q_{ij\ell}$ of $H$ containing exactly the three cliques
\[
S_{ij},\ S_{i\ell},\ S_{j\ell}
\]
from this family.
Then there exist a $(k-3)$-clique $T$ of $H$ and pairwise distinct vertices $x_1,\dots,x_n\in V(H)\setminus T$ such that
\[
S_{ij}=T+x_i+x_j\qquad (1\le i<j\le n).
\]
\end{lemma}

\begin{proof}
We regard cliques as their vertex sets. First consider $Q_{123}$.
The three distinct $(k-1)$-subcliques $S_{12},S_{13},S_{23}$ of the
$k$-set $Q_{123}$ are obtained by deleting three distinct vertices.
Thus, there are a $(k-3)$-set $T$ and distinct vertices
$x_1,x_2,x_3\notin T$ such that
\[
S_{12}=T+x_1+x_2,\qquad
S_{13}=T+x_1+x_3,\qquad
S_{23}=T+x_2+x_3.
\]

Fix $\ell\ge 4$. Write
\[
Q_{12\ell}=S_{12}+a,\qquad Q_{13\ell}=S_{13}+b
\]
with $a\notin S_{12}$ and $b\notin S_{13}$. The exactness condition for
$Q_{12\ell}$ gives $a\ne x_3$, since otherwise $Q_{12\ell}=Q_{123}$ would
also contain the family members $S_{13}$ and $S_{23}$. Similarly
$b\ne x_2$. Since $S_{1\ell}$ is contained in both $Q_{12\ell}$ and
$Q_{13\ell}$, their intersection has size at least $k-1$. But
\[
Q_{12\ell}=T+x_1+x_2+a,\qquad
Q_{13\ell}=T+x_1+x_3+b.
\]
These two $k$-sets already share $T+x_1$, of size $k-2$. Because
$x_2\ne x_3$, $a\ne x_3$, and $b\ne x_2$, the only way their intersection
has one more vertex is that $a=b$. Set $x_\ell:=a=b$. Then
\[
Q_{12\ell}\cap Q_{13\ell}=T+x_1+x_\ell,
\]
and hence
\[
S_{1\ell}=T+x_1+x_\ell.
\]

Now write $Q_{23\ell}=S_{23}+c$ with $c\notin S_{23}$. The exactness
condition for $Q_{23\ell}$ gives $c\ne x_1$, since otherwise
$Q_{23\ell}=Q_{123}$ would also contain $S_{12}$ and $S_{13}$. Since
$S_{2\ell}$ is contained in both $Q_{12\ell}$ and $Q_{23\ell}$, the same
intersection argument applied to
\[
Q_{12\ell}=T+x_1+x_2+x_\ell,\qquad
Q_{23\ell}=T+x_2+x_3+c
\]
forces $c=x_\ell$. Therefore,
\[
S_{2\ell}=Q_{12\ell}\cap Q_{23\ell}=T+x_2+x_\ell.
\]
Finally,
\[
Q_{13\ell}\cap Q_{23\ell}=T+x_3+x_\ell,
\]
so the common $(k-1)$-subclique $S_{3\ell}$ is
\[
S_{3\ell}=T+x_3+x_\ell.
\]

At this point, for every $\ell\ge 4$ we have
$x_\ell\notin T\cup\{x_1,x_2,x_3\}$: indeed $x_\ell=a\notin S_{12}$
and the exactness argument above gives $x_\ell\ne x_3$. If
$4\le \ell<m\le n$ and $x_\ell=x_m$, then
\[
S_{1\ell}=T+x_1+x_\ell=T+x_1+x_m=S_{1m},
\]
contrary to the assumed distinctness of the family. Hence,
$x_1,\dots,x_n$ are pairwise distinct.

It remains to handle pairs $4\le i<j\le n$. The clique $Q_{1ij}$
contains both
\[
S_{1i}=T+x_1+x_i,\qquad S_{1j}=T+x_1+x_j,
\]
whose union is therefore the $k$-set $T+x_1+x_i+x_j$. Hence,
\[
Q_{1ij}=T+x_1+x_i+x_j.
\]
Similarly,
\[
Q_{2ij}=T+x_2+x_i+x_j.
\]
The set $S_{ij}$ is contained in both $Q_{1ij}$ and $Q_{2ij}$, so
\[
S_{ij}\subseteq Q_{1ij}\cap Q_{2ij}=T+x_i+x_j.
\]
Both sides have size $k-1$, and therefore
\[
S_{ij}=T+x_i+x_j.
\]

The displayed formula now holds for every pair $i<j$. Finally,
$T\subseteq S_{12}$, so $T$ is a $(k-3)$-clique of $H$.
\end{proof}

We now extend the common-core normalization to arbitrary $s$.
\begin{lemma}\label{lem:camera-TJ-general-s-star-common-core}
Let $s\ge 2$, let $H$ be a graph, let $k\ge s+1$, let $n\ge s+2$, and let
\[
\{S_X:X\in \tbinom{[n]}{s}\}
\]
be distinct $(k-1)$-cliques of $H$ such that for every $(s+1)$-subset $A\subseteq [n]$ there exists a $k$-clique $Q_A$ of $H$ containing exactly the cliques
\[
\{S_X:X\in \tbinom{A}{s}\}
\]
from this family.
Then there exist a $(k-s-1)$-clique $T$ of $H$ and pairwise distinct vertices $x_1,\dots,x_n\in V(H)\setminus T$ such that
\[
S_X=T+\sum_{i\in X}x_i\qquad (X\in \tbinom{[n]}{s}).
\]
\end{lemma}

\begin{proof}
We argue by induction on $s$.
For $s=2$, this is exactly \cref{lem:camera-TJ-pair-star-common-core}.
Assume now that the statement holds for this value of $s$, and let
\[
\{S_X:X\in \tbinom{[n]}{s+1}\}
\]
be a family satisfying the hypotheses with $s+1$ in place of $s$.
Set
\[
A_0:=\{1,2,\dots,s+2\}.
\]
The $k$-clique $Q_{A_0}$ contains the $s+2$ distinct $(k-1)$-cliques
\[
\{S_{A_0\setminus\{a\}}:a\in A_0\}.
\]
Thus, there exist a $(k-s-2)$-clique $T$ and distinct vertices $x_1,\dots,x_{s+2}\notin T$ such that
\[
S_{A_0\setminus\{a\}}=T+\sum_{i\in A_0\setminus\{a\}}x_i\qquad (a\in A_0).
\]
For every $\displaystyle Y\in \binom{\{2,\dots,n\}}{s}$, define
\[
U_Y:=S_{\{1\}\cup Y}.
\]
This anchored family is distinct, since it is a subfamily of the original family.
Let $\displaystyle B\in \binom{\{2,\dots,n\}}{s+1}$.
By the full-family exactness hypothesis applied to $\{1\}\cup B$, the members of the original family contained in $Q_{\{1\}\cup B}$ are precisely
\[
\{S_B\}\cup \{S_{\{1\}\cup Y}:Y\in \tbinom{B}{s}\}
=
\{S_B\}\cup \{U_Y:Y\in \tbinom{B}{s}\}.
\]
Thus, after restricting to the anchored subfamily $\{U_Y\}$, the same $k$-clique $Q_{\{1\}\cup B}$ contains exactly the cliques
\[
\{U_Y:Y\in \tbinom{B}{s}\}
\]
from that anchored subfamily.
These are precisely the hypotheses of the induction statement for parameter $s$ on the ground set $\{2,\dots,n\}$, after relabelling this ground set as $[n-1]$; here $k\ge s+2\ge s+1$ and $n-1\ge s+2$.
Hence, the induction hypothesis applies to the family $\{U_Y\}$.
We obtain a $(k-s-1)$-clique $T'$ and pairwise distinct vertices $y_2,\dots,y_n\in V(H)\setminus T'$ such that
\[
U_Y=T'+\sum_{i\in Y}y_i\qquad (Y\in \tbinom{\{2,\dots,n\}}{s}).
\]
For each $a\in\{2,\dots,s+2\}$, let $Y_a:=A_0\setminus\{1,a\}$.
Then $U_{Y_a}=S_{A_0\setminus\{a\}}$.
Intersecting these $s+1$ cliques gives $T'$ on the one hand, and $T+x_1$ on the other.
Therefore,
\[
T'=T+x_1.
\]
Also
\[
Q_{A_0}=T+x_1+\sum_{i=2}^{s+2}x_i=T'+\sum_{i=2}^{s+2}x_i.
\]
For each $a\in\{2,\dots,s+2\}$, the clique $U_{Y_a}$ is a $(k-1)$-subclique of $Q_{A_0}$.
On the $x$-description we have
\[
U_{Y_a}=Q_{A_0}\setminus\{x_a\}\qquad (a=2,\dots,s+2),
\]
while the $y$-description gives
\[
U_{Y_a}=T'+\sum_{i\in\{2,\dots,s+2\}\setminus\{a\}}y_i\qquad (a=2,\dots,s+2).
\]
Hence,
\[
\bigcup_{a=2}^{s+2}U_{Y_a}=T'+\sum_{i=2}^{s+2}y_i.
\]
Since each $U_{Y_a}\subseteq Q_{A_0}$ and every vertex of $Q_{A_0}\setminus T'$ belongs to some $U_{Y_a}$, it follows that
\[
Q_{A_0}=T'+\sum_{i=2}^{s+2}y_i.
\]
Therefore, for each $a=2,\dots,s+2$, the same $(k-1)$-subclique $U_{Y_a}$ omits $x_a$ in the $x$-description and omits $y_a$ in the $y$-description.
Since, inside a fixed clique, each $(k-1)$-subclique is uniquely determined by the unique vertex it omits, we conclude
\[
x_a=y_a\qquad (a=2,\dots,s+2).
\]
For every $i\ge s+3$, define $x_i:=y_i$.
Then the anchored formula becomes
\[
S_{\{1\}\cup Y}=T+x_1+\sum_{i\in Y}x_i\qquad (Y\in \tbinom{\{2,\dots,n\}}{s}).
\]
Now let $\displaystyle X\in \binom{\{2,\dots,n\}}{s+1}$.
The $k$-clique $Q_{\{1\}\cup X}$ contains the $s+1$ distinct $(k-1)$-cliques
\[
\Bigl\{S_{\{1\}\cup (X\setminus\{a\})}:a\in X\Bigr\}.
\]
Each of them equals
\[
T+x_1+\sum_{i\in X\setminus\{a\}}x_i.
\]
The union of any two such cliques already equals
\[
T+x_1+\sum_{i\in X}x_i,
\]
which has size $k$.
Therefore,
\[
Q_{\{1\}\cup X}=T+x_1+\sum_{i\in X}x_i.
\]
Inside this $k$-clique, the displayed $s+1$ $(k-1)$-subcliques are exactly those omitting one of the vertices $x_a$ with $a\in X$.
Since $Q_{\{1\}\cup X}$ contains exactly the $s+2$ family members
\[
\{S_Y:Y\in \tbinom{\{1\}\cup X}{s+1}\},
\]
the remaining family member $S_X$ is a $(k-1)$-subclique of $Q_{\{1\}\cup X}$ distinct from the displayed $s+1$ cliques.
Hence, there exists a unique vertex $m_X\in T\cup\{x_1\}$ such that
\[
S_X=\Bigl(T+x_1+\sum_{i\in X}x_i\Bigr)-m_X.
\]

We claim that $m_X$ is independent of $X$.
Let $\displaystyle X,Y\in \binom{\{2,\dots,n\}}{s+1}$ with $|X\cap Y|=s$, and set
\[
A:=X\cup Y.
\]
Then $|A|=s+2$, so both $S_X$ and $S_Y$ are contained in the $k$-clique $Q_A$.
If $m_X\ne m_Y$, then
\[
S_X\cup S_Y=T+x_1+\sum_{i\in A}x_i,
\]
which has size
\[
|T|+1+|A|=(k-s-2)+1+(s+2)=k+1,
\]
contrary to $S_X\cup S_Y\subseteq Q_A$.
Therefore, $m_X=m_Y$ whenever $|X\cap Y|=s$.
Because $n\ge s+3$, the Johnson graph on $\displaystyle \binom{\{2,\dots,n\}}{s+1}$ is connected.
Hence, all $m_X$ are equal to one common vertex $m\in T\cup\{x_1\}$.

Set
\[
\widetilde T:=(T+x_1)-m,\qquad
\widetilde x_1:=m,\qquad
\widetilde x_i:=x_i\quad (i=2,\dots,n).
\]
The set $T+x_1$ is a clique because it equals $T'$.
Hence, $\widetilde T$ is a $(k-s-2)$-clique.
The vertices $\widetilde x_1,\dots,\widetilde x_n$ are pairwise distinct and lie outside $\widetilde T$.
Moreover,
\[
S_{\{1\}\cup Y}=\widetilde T+\widetilde x_1+\sum_{i\in Y}\widetilde x_i
\qquad
\bigl(Y\in \tbinom{\{2,\dots,n\}}{s}\bigr),
\]
and
\[
S_X=\widetilde T+\sum_{i\in X}\widetilde x_i
\qquad
\bigl(X\in \tbinom{\{2,\dots,n\}}{s+1}\bigr).
\]
Renaming $\widetilde T,\widetilde x_1,\dots,\widetilde x_n$ as $T,x_1,\dots,x_n$, the desired identity holds for all $(s+1)$-subsets.
This induction completes the proof.
\end{proof}

We can now exclude the stable middle range on the TJ side.
The proof begins in the same way, but instead of using the star-count bound immediately, it feeds the resulting star family into the common-core normalization.
\begin{theorem}\label{thm:camera-tj-no-middle}
Let $n>2r\ge 6$.
If
\[
r+1\le k\le n-r-1,
\]
then
\[
k\notin \KTJ(J(n,r)).
\]
\end{theorem}

\begin{proof}
Suppose to the contrary that $J(n,r)\cong \TJ{k}{H}$ for some graph $H$, where $n>2r\ge 6$ and $r+1\le k\le n-r-1$.
Fix an isomorphism $\phi:J(n,r)\to\TJ{k}{H}$.
For every $(r-1)$-subset $X\subseteq [n]$, let
\[
\mathcal M_X:=\{X\cup\{y\}:y\in [n]\setminus X\}
\]
be the corresponding Johnson $(r-1)$-star.
By \cref{lem:camera-JnR-max-stars-stable}, these are exactly the maximum cliques of $J(n,r)$, each of size $n-r+1$.
Fix one such clique $\mathcal M_X$.
Since $k\le n-r-1$, we have
\[
n-r+1\ge k+2>k+1,
\]
while every top clique of $\TJ{k}{H}$ has size exactly $k+1$ by \cref{cor:camera-TJ-maximal-cliques}. Because isomorphisms preserve maximal cliques, the $\phi$-image of $\mathcal M_X$ is a maximal clique of $\TJ{k}{H}$ and hence is either a star clique or a top clique. Therefore, the $\phi$-image of $\mathcal M_X$ cannot be a top clique. Since $\mathcal M_X$ was arbitrary, every Johnson maximum clique is realized as a star clique in the witness.
Thus, for every $(r-1)$-subset $X\subseteq [n]$ there exists a $(k-1)$-clique $S_X$ of $H$ such that
\[
\phi(\mathcal M_X)=\mathcal C_H(S_X).
\]
Distinct Johnson stars have distinct images under the isomorphism, and each such non-singleton star clique is determined by its core, namely the intersection of all its vertices.
Hence, the cliques $S_X$ are pairwise distinct.

For every $r$-subset $A\subseteq [n]$, let $Q_A:=\phi(A)$ be the $k$-clique of $H$ representing the vertex $A$ of $J(n,r)$.
If $X\subseteq A$ and $|X|=r-1$, then the Johnson vertex $A$ lies in the Johnson star $\mathcal M_X$.
Therefore, $Q_A$ belongs to the star clique $\mathcal C_H(S_X)$, which means that $S_X\subseteq Q_A$.
Hence, $Q_A$ contains the $r$ cliques
\[
\mathcal F_A:=\{S_X:X\in \tbinom{A}{r-1}\}.
\]
These are exactly the members of the family $\displaystyle \{S_X:X\in \binom{[n]}{r-1}\}$ contained in $Q_A$, because the vertex $A$ of $J(n,r)$ lies in exactly the $r$ Johnson stars indexed by its $(r-1)$-subsets.
This family is precisely the configuration covered by \cref{lem:camera-TJ-general-s-star-common-core} with $s=r-1$.
We obtain a $(k-r)$-clique $T$ and pairwise distinct vertices $x_1,\dots,x_n$ such that
\[
S_X=T+\sum_{i\in X}x_i\qquad (X\in \tbinom{[n]}{r-1}).
\]

Now fix an $r$-subset $A\subseteq [n]$.
For each $a\in A$, the clique $Q_A$ contains
\[
S_{A\setminus\{a\}}=T+\sum_{i\in A\setminus\{a\}}x_i.
\]
If $a,b\in A$ are distinct, then the union of these two contained subcliques is
\[
S_{A\setminus\{a\}}\cup S_{A\setminus\{b\}}=T+\sum_{i\in A}x_i.
\]
This union has size
\[
|T|+|A|=(k-r)+r=k.
\]
Since $Q_A$ is itself a $k$-clique containing that union, we obtain
\[
Q_A=T+\sum_{i\in A}x_i.
\]
It follows that every $r$-subset of $\{x_1,\dots,x_n\}$ is a clique in $H$.
Indeed, for any such $r$-subset $A$, the vertices $T\cup\{x_i:i\in A\}$ form the clique $Q_A$.
Because $n>2r$, every pair $x_i,x_j$ is contained in some $r$-subset of $[n]$, so the vertices $x_1,\dots,x_n$ are pairwise adjacent.
Moreover, every vertex of $T$ is adjacent to each $x_i$. To see this, choose an $(r-1)$-subset $X\subseteq [n]$ with $i\in X$; then both the chosen vertex of $T$ and $x_i$ belong to the clique
\[
S_X=T+\sum_{h\in X}x_h.
\]
Hence,
\[
H[T\cup\{x_1,\dots,x_n\}]\cong K_{n+k-r}.
\]
Since $k\ge r+1$, the clique $T$ is nonempty.
Therefore, every $k$-subset of $\{x_1,\dots,x_n\}$ is a $k$-clique of $H$ distinct from every $Q_A$, because each $Q_A$ contains the nonempty set $T$.
Hence,
\[
|V(\TJ{k}{H})|\ge \binom{n}{r}+\binom{n}{k}>\binom{n}{r}=|V(J(n,r))|,
\]
a contradiction.
Thus, no such $k$ exists.
\end{proof}

We next record the Johnson clique number in the range needed for the large-$k$ exclusion.
It follows from the same star/top classification of cliques in Johnson graphs~\cite[Cor.~3.5]{ShuldinerO22}.
\begin{lemma}\label{lem:camera-JnR-clique-number-half}
Let $n\ge 2r\ge 4$.
Then
\[
\omega(J(n,r))=n-r+1.
\]
\end{lemma}

The next theorem excludes all larger values.
Here the argument changes: once $k>n-r$, top cliques are too large, so we use the neighborhood structure instead of the maximum-clique template.
\begin{theorem}\label{thm:camera-tj-large}
Let $n\ge 2r\ge 4$.
If $k>n-r$, then
\[
k\notin \KTJ(J(n,r)).
\]
\end{theorem}

\begin{proof}
Suppose to the contrary that $J(n,r)\cong \TJ{k}{H}$ for some $k>n-r$.
A top clique in $\TJ{k}{H}$ has size $k+1$.
By \cref{lem:camera-JnR-clique-number-half}, we have $\omega(J(n,r))=n-r+1$.
Thus, the inequality $k>n-r$ implies $k+1>\omega(J(n,r))$.
If $H$ contained a $(k+1)$-clique $U$, then $\mathcal T_H(U)$ would be a top clique of size $k+1$ in $\TJ{k}{H}$, contradicting this clique-number bound.
Hence, $H$ contains no $(k+1)$-clique.
Fix a vertex $Q$ of $\TJ{k}{H}$.
Its neighbors partition according to the deleted vertex of $Q$, and each part is a clique.
If two neighbors from distinct parts were adjacent, then they would differ in exactly one vertex. Since they delete different vertices of $Q$, this can happen only if they insert the same outside vertex. That outside vertex would then be adjacent to all of $Q$, and $H$ would contain a $(k+1)$-clique, impossible.
Hence, every neighborhood in $\TJ{k}{H}$ is a disjoint union of cliques.
On the other hand, every vertex neighborhood in $J(n,r)$ is isomorphic to $L(K_{r,n-r})$ by \cref{lem:camera-jnr-neighborhood-line}, and because $r,n-r\ge 2$, this line graph contains an induced $P_3$.
So it is not a disjoint union of cliques.
This contradiction proves the claim.
\end{proof}

These ingredients already settle the stable higher-rank regime.
\begin{corollary}\label{cor:camera-tj-stable-higher-rank-all}
If $n>2r\ge 6$, then
\[
\KTJ(J(n,r))=\{r,n-r\}.
\]
\end{corollary}

\begin{proof}
The values $r$ and $n-r$ are feasible because
\[
\TJ{r}{K_n}\cong J(n,r)
\quad\text{and}\quad
\TJ{n-r}{K_n}\cong J(n,n-r)\cong J(n,r).
\]
Conversely, because $J(n,r)$ is not complete, \cref{prop:TJ1} excludes $k=1$.
By \cref{cor:camera-JnR-no-small-TJ}, no value in $\{2,3,\dots,r-1\}$ is feasible.
If $k\notin\{r,n-r\}$, then either
\[
r+1\le k\le n-r-1
\]
or
\[
k>n-r.
\]
The first case contradicts \cref{thm:camera-tj-no-middle}, and the second contradicts \cref{thm:camera-tj-large}.
Therefore, only $r$ and $n-r$ remain feasible.
\end{proof}

We first determine the maximum cliques in the boundary case.
The subsequent low-$k$ exclusion again uses the star-count contradiction, now with the boundary count $2r$.
This boundary count is a direct incidence count from the standard star/top classification of maximum cliques in Johnson graphs~\cite[Thm.~3.4 and Cor.~3.5]{ShuldinerO22}.
\begin{lemma}\label{lem:camera-JnR-max-cliques-boundary}
Assume $n=2r\ge 4$.
Then every maximum clique of $J(2r,r)$ is either a Johnson star clique or a Johnson top clique, and every such clique has size $r+1$.
Moreover, every vertex of $J(2r,r)$ lies in exactly $2r$ maximum cliques.
\end{lemma}

This corollary yields the low-$k$ exclusion in the boundary case.
\begin{corollary}\label{cor:camera-JnR-no-small-TJ-boundary}
For every $r\ge 2$,
\[
\{1,\dots,r-1\}\cap \KTJ(J(2r,r))=\varnothing.
\]
\end{corollary}

\begin{proof}
Suppose to the contrary that $J(2r,r)\cong \TJ{k}{H}$ for some $1\le k\le r-1$.
Fix an isomorphism $\phi:J(2r,r)\to\TJ{k}{H}$.
Because $J(2r,r)$ is not complete for $r\ge 2$, \cref{prop:TJ1} gives $k\ne 1$.
Hence, $2\le k\le r-1$.
If $r=2$, this is already impossible; hence assume $r\ge 3$.
By \cref{lem:camera-JnR-max-cliques-boundary}, every maximum clique of $J(2r,r)$ has size $r+1$, and every vertex lies in exactly $2r$ such maximum cliques. Because isomorphisms preserve maximal cliques, the $\phi$-image of each maximum clique of $J(2r,r)$ is a maximal clique of $\TJ{k}{H}$ and hence is either a star clique or a top clique by \cref{cor:camera-TJ-maximal-cliques}. A top clique in $\TJ{k}{H}$ has size exactly $k+1$, and
\[
k+1\le r<r+1.
\]
Therefore, no maximum clique of $J(2r,r)$ can be a top clique in the witness, so every maximum clique must be realized as a star clique.
Fix a vertex $A$ of $J(2r,r)$.
The vertex $\phi(A)$ of $\TJ{k}{H}$ therefore lies in at least $2r$ distinct star cliques.
But this contradicts \cref{lem:camera-TJ-clique-at-most-k-stars}, because $2r>k$.
Therefore, no such $k$ exists.
\end{proof}

We can now settle the boundary TJ formula.
\begin{theorem}\label{thm:camera-tj-boundary}
For every $r\ge 2$,
\[
\KTJ(J(2r,r))=\{r\}.
\]
\end{theorem}

\begin{proof}
The inclusion $\{r\}\subseteq \KTJ(J(2r,r))$ is immediate from the complete-graph witness:
\[
\TJ{r}{K_{2r}}\cong J(2r,r).
\]
For the reverse inclusion, let $k\in \KTJ(J(2r,r))$.
Because $J(2r,r)$ is not complete for $r\ge 2$, \cref{prop:TJ1} gives $k\ne 1$.
If $k<r$, then \cref{cor:camera-JnR-no-small-TJ-boundary} gives a contradiction.
If $k>r$, then the special case $n=2r$ of \cref{thm:camera-tj-large} gives a contradiction.
Hence, $k=r$.
\end{proof}

We are now ready to prove \Cref{thm:camera-tj-johnson}.
\begin{proof}[Proof of \Cref{thm:camera-tj-johnson}]
If $s=1$, then $J(n,r)\cong J(n,1)$ by Johnson duality, so the claim follows from \cref{cor:camera-johnson-complete-levels-tj}.
Assume from now on that $s\ge 2$.
By Johnson duality, we may replace $r$ by $s$ and therefore assume $r=s\le n/2$.
If $n=2s$, then the claim is exactly \cref{thm:camera-tj-boundary}.
If $n>2s$ and $s=2$, then the claim is exactly \cref{thm:camera-KTJ-Jn2}.
If $n>2s$ and $s\ge 3$, then the claim is exactly \cref{cor:camera-tj-stable-higher-rank-all}.
These cases cover all remaining possibilities.
\end{proof}

\section*{Acknowledgements}
The research was supported by Vietnam National University, Hanoi under the project QG.25.07 ``A study on reconfiguration problems from algorithmic and graph-theoretic perspectives''.


\printbibliography


\end{document}